\documentclass[11pt]{amsart}

\usepackage{amssymb}
\usepackage{amsmath}
\usepackage{amsthm}
\usepackage{color} 
\usepackage{verbatim}

\newtheorem{theorem}{Theorem}[section]
\newtheorem{lemma}{Lemma}[section]
\newtheorem{remark}{Remark}[section]
\newtheorem{assumption}{Assumption}[section]

\numberwithin{equation}{section}

\newcommand{\Rd}{\mathbb{R}^d}
\newcommand{\ve}{\varepsilon}

\newcommand{\psk}{\psi_{(\sigma^k)}}
\newcommand{\mypsi}{\psi_{(\sigma^i)}}

\newcommand{\pxsop}{\frac{\psi_{(\xi)}^2}{\psi}}

\newcommand{\pxsops}{\frac{\psi_{(\xi)}^2}{\psi^2}}
\newcommand{\pxtsops}{\frac{\psi_{(\xi_t)}^2}{\psi^2}}

\begin{document}

\title[Regularity of fully nonlinear elliptic PDEs in domains]{A probabilistic approach to interior regularity of fully nonlinear degenerate elliptic equations in smooth domains}
\author{Wei Zhou}
\address{127 Vincent Hall, 206 Church St. SE, Minneapolis, MN 55455}

\maketitle
\let\thefootnote\relax
\let\thefootnote\relax\footnotetext{Institutional Affiliation: School of Mathematics, University of Minnesota}
\let\thefootnote\relax\footnotetext{Mailing Address: 127 Vincent Hall, 206 Church St. SE, Minneapolis, MN 55455}
\let\thefootnote\relax\footnotetext{Email Address: zhoux123@umn.edu}
\let\thefootnote\relax\footnotetext{Phone: 1-612-625-3412}
\begin{abstract}
\noindent
We consider the value function of a stochastic optimal control of degenerate diffusion processes in a domain $D$. We study the smoothness of the value function, under the assumption of the non-degeneracy of the diffusion term along the normal to the boundary and an interior condition weaker than the non-degeneracy of the diffusion term. When the diffusion term, drift term, discount factor, running payoff and terminal payoff are all in the class of $C^{1,1}(\bar D)$, the value function turns out to be the unique solution in the class of $C_{loc}^{1,1}(D)\cap C^{0,1}(\bar D)$ to the associated degenerate Bellman equation with Dirichlet boundary data. Our approach is probabilistic.
\end{abstract}

\section{Introduction}

We consider the Dirichlet problem for the Bellman equation 
\begin{equation}
\left\{
\begin{array}{rcll}
\displaystyle\sup_{\alpha\in A}\big[L^\alpha v(x)-c(\alpha, x)v(x)+f(\alpha,x)\big]&=&0 &\text{in } D\\
v&=&g &\text{on }\partial D,
\end{array}
\right.  \label{1c}
\end{equation}
where $L^\alpha v(x):=a^{ij}(\alpha, x)v_{x^ix^j}(x)+ b^i(\alpha, x)v_{x^i}(x)$, and summation convention of repeated indices is understood. 
On the one hand, it is known that under appropriate conditions the Dirichlet problem for the fully nonlinear convex elliptic equation
\begin{equation}
\left\{
\begin{array}{rcll}
F\big(v_{x^ix^j}(x), v_{x^{i}}(x),v(x),x\big)&=&0&\text{in } D\\
v&=&g&\text{on }\partial D
\end{array}
\right.
\end{equation}
can be rewritten as a Bellman equation in the form of (\ref{1c}).
On the other hand, under suitable regularity assumptions on $a, b, c, f, g$ and $D$, the Bellman equation (\ref{1c}) is satisfied by the value function
\begin{equation}\label{1d}
v(x)=\sup_{\alpha\in\mathfrak{A}}v^\alpha(x),
\end{equation}
where
\begin{equation}
v^\alpha(x)=E\bigg[g\big(x^{\alpha,x}_{\tau^{\alpha,x}}\big)e^{-\phi^{\alpha,x}_{\tau^{\alpha,x}}}+\int_0^{\tau^{\alpha,x}}f^{\alpha_s}\big(x_s^{\alpha,x}\big)e^{-\phi_s^{\alpha,x}}ds\bigg]   ,
\end{equation}
$$\text{with }\phi_t^{\alpha,x}=\int_0^t c^{\alpha_s}(x_s^{\alpha,x})ds,$$
in a control problem associated with the family of It\^o equations
\begin{equation}
x_t^{\alpha,x}=x+\int_0^t\sigma^{\alpha_s}(x_s^{\alpha,x})dw_s+\int_0^tb^{\alpha_s}(x_s^{\alpha,x})ds,
\end{equation}
where $\tau^{\alpha,x}$ is the first exit time of $x_t^{\alpha,x}$ from $D$.

However, in general, $v$ defined by (\ref{1d}) is not sufficiently smooth, or even continuous, so $v$ in (\ref{1d}) is known as a probabilistic solution to (\ref{1c}). We are interested in understanding under what conditions, $v$ given by (\ref{1d}) is twice differentiable and is the unique solution of (\ref{1c}) in an appropriate sense. The main difficulties in dealing with this problem are the fully nonlinearity, the degeneracy of the operator, the infiniteness of the time horizon  and the non-vanishing boundary condition.

The results stated and proved here are closely related to those obtained by M.~V.~Safonov \cite{MR0445110} (1977), \cite{MR496596} (1978); P.-L.~Lions \cite{MR725360} (1983) and N.~V.~Krylov \cite{MR992979} (1989). In \cite{MR0445110} and \cite{MR496596}, the domain $D$ is two-dimensional, and the arguments are based on the fact that the controlled processes are in a plane region. In \cite{MR725360}, the regularity results are proved by a combination of probabilistic and PDE arguments, which heavily rely on the assumption that the discount coefficient $c^\alpha(x)$ is sufficiently large to bound first derivatives of $\sigma^\alpha(x)$ and $b^\alpha(x)$. In \cite{MR992979}, the boundary data $g$ is assumed to be of class $C^4$, and under certain assumptions, it is proved that $v$ has second derivatives bounded up to the boundary. The results are obtained in a purely probabilistic approach by introducing and using quasiderivatives and a reduction of controlled processes in a domain to controlled processes on a surface without boundary in the space having four more dimensions. 

In this article, under a more general setting, we give sufficient conditions under which the first and second derivatives of $v$ given by (\ref{1d}) exist almost everywhere in $D$, which implies the existence and uniqueness for the associated Dirichlet problem (\ref{1c}). Moreover, since we assume that the boundary data $g\in C^{k-1,1}(\bar D)$ when we investigate the existence of the $k$-th order derivatives of $v$, where $k=1,2$, the derivatives of $v$, if they do exist (a.e.), may not be bounded up to the boundary. Therefore, we also estimate the first and second derivatives. 

The main result is stated in Section 2, and the proof is given in Section 3. Our approach is  probabilistic by using quasiderivatives. However, to deal with the boundary, instead of adding four more dimensions, we construct two families of local supermartingales to bound the moments of quasiderivatives near the boundary and in the interior of the domain, respectively. For the background and motivations of quasiderivative method, we refer to \cite{MR2144644,quasi-linear} and the references therein.

To conclude this section, we introduce the notation: For $k=1,2$, let $C^k(\bar{D})$ be the space of $k$-times continuously differentiable functions in $\bar{D}$ with finite norm given by
\[ |g|_{1,D}=|g|_{0,D}+|g_x|_{0,D},\ \  |g|_{2,D}=|g|_{1,D}+|g_{xx}|_{0,D},\]
respectively, where 
\[|g|_{0,D}=\sup_{x\in D}|g(x)|,\]
$g_x$ is the gradient vector of $g$, and $g_{xx}$ is the Hessian matrix of $g$. For $\beta\in(0,1]$, the H\"older spaces $C^{k,\beta}(\bar D)$ are defined as the subspaces of $C^k(\bar D)$ consisting of functions with finite norm
$$|g|_{k,\beta,D}=|g|_{k,D}+[g]_{\beta,D},\ \ \mbox{ with }[g]_{\beta, D}=\sup_{x,y\in D}\frac{|g(x)-g(y)|}{|x-y|^\beta}.$$
$\Rd$ is the $d$-dimensional Euclidean space with $x = (x^1, x^2, . . . , x^d )$ representing a typical point in $\Rd$, and $(x, y) =\sum_{i=1}^d x^iy^i$ is the inner product for $x, y \in \Rd$. For $x,y,z\in\Rd$, set
\begin{align*}
u_{(y)}=&\sum_{i=1}^d u_{x^i}y^i,\ \  u_{(y)(z)}=\sum_{i,j=1}^d u_{x^i x^j}y^i z^j,
\end{align*}
$$u_{(y)}^2=(u_{(y)})^2.$$
For any matrix $\sigma=(\sigma^{ij})$,
$$\|\sigma\|^2:=\mathrm{tr}\sigma\sigma^*=\sum_{i,j}(\sigma^{ij})^2.$$
We also use the notation
\[s\wedge t=\min(s,t),\ \  s \vee t=\max(s,t).\]
Constants $K, M$ and $N$ appearing in inequalities are usually not indexed. They may differ even in the same chain of inequalities.

\section{Main results}

Assume that $(\Omega,\mathcal{F},P)$ is a complete probability space and $\{\mathcal{F}_t;t \ge 0\}$ an increasing filtration of $\sigma $-algebras $\mathcal{F}_t \subset \mathcal{F}$ which are complete with respect to $\mathcal{F},P$. Let $(w_t ,\mathcal{F}_t ;t \ge 0)$ be a $d_1$-dimensional Wiener process on $(\Omega,\mathcal{F},P)$.

Let $A$ be a separable metric space. Suppose that the following have been defined for each $\alpha \in A$ and $x \in \Rd$: a $d \times d_1$ matrix $\sigma^\alpha(x)$, a $d$-dimensional vector $b^\alpha(x)$ and real scalars $c^\alpha(x)\geq0$ and $f^\alpha(x)$. We assume that $\sigma$, $b$, $c$ and $f$ are Borel measurable on $A\times \Rd$, and $g(x)$ is a Borel measurable function on $\Rd$. We also assume that $\sigma^\alpha$ , $b^\alpha$, $c^\alpha$ and their first and second derivatives are all continuous in $x$ uniformly with respect to $\alpha$.

Let $D\in C^4$ be a bounded domain in $\Rd$, then there exists a function $\psi\in C^4$ satisfying
$$
\psi>0 \mbox{ in }D,\ \ \psi=0\mbox{ and } |\psi_x|\ge1 \mbox{ on }\partial D.
$$
Additionally, we assume that
$$
\sup_{\alpha\in A}L^\alpha\psi\le-1\mbox{ in }D,
$$
with
$$
L^\alpha:=(a^\alpha)^{ij}(x)\frac{\partial^2}{\partial x^i\partial x^j}+(b^\alpha)^i(x)\frac{\partial}{\partial x^i},
$$
where $a=1/2(\sigma\sigma^*)$. We also assume that 
\begin{equation}\label{tiao}
|(\sigma^\alpha)^{ij}|_{2,D}+|(b^\alpha)^i|_{2,D}+|c^\alpha|_{2,D}+|\psi|_{4,D}\le K_0,
\end{equation}
$$
\forall \alpha\in A, 1\le i\le d, 1\le j\le d_1,
$$
with $K_0\in[1,\infty)$, not depending on $\alpha$.

By $\mathfrak{A}$, we denote the set of all functions 
$\alpha_r(\omega)$ on $\Omega\times[0,\infty)$ which are $\mathcal{F}_r$-adapted and measurable in $(\omega, r)$ with values in $A$.

For $\alpha\in\mathfrak{A}$ and $x\in D$, we consider the It\^o equation
\begin{equation}\label{itox}
x_t^{\alpha,x}=x+\int_0^t\sigma^{\alpha_s}(x_s^{\alpha,x})dw_s+\int_0^tb^{\alpha_s}(x_s^{\alpha,x})ds.
\end{equation}
The solution of this equation is known to exist and to be unique by our assumptions on $\sigma^\alpha$ and $b^\alpha$.

Let $\tau^{\alpha,x}$ be the first exit time of $x_t^{\alpha,x}$ from $D$:
$$
\tau^{\alpha,x}=\inf\{t\ge0:x_t^{\alpha,x}\notin D\}.
$$

For any $t\ge 0$, we define
$$
\phi_t^{\alpha,x}=\int_0^tc^{\alpha_s}(x_s^{\alpha,x})ds.
$$

Set
\begin{equation}
v(x)=\sup_{\alpha\in\mathfrak{A}}v^\alpha(x),\label{v}
\end{equation}
with
\begin{equation}\label{vax}
v^\alpha(x)=E^\alpha_x\bigg[g\big(x_\tau\big)e^{-\phi_{\tau}}+\int_0^{\tau}f^{\alpha_s}\big(x_s\big)e^{-\phi_s}ds\bigg] , 
\end{equation}
where we use common abbreviated notation, according to which we put the indices $\alpha$ and $x$ beside the expectation sign instead of explicitly exhibiting them inside the expectation sign for every object that can carry all or part of them. Namely, 
$$\begin{gathered}
E^\alpha_x\bigg[g\big(x_\tau\big)e^{-\phi_{\tau}}+\int_0^{\tau}f^{\alpha_s}\big(x_s\big)e^{-\phi_s}ds\bigg]\\
=E\bigg[g\big(x^{\alpha,x}_{\tau^{\alpha,x}}\big)e^{-\phi^{\alpha,x}_{\tau^{\alpha,x}}}+\int_0^{\tau^{\alpha,x}}f^{\alpha_s}\big(x_s^{\alpha,x}\big)e^{-\phi^{\alpha,x}_s}ds\bigg] .
\end{gathered}$$

The value function $v(x)$ given by (\ref{v}) and (\ref{vax}) is the probabilistic solution of the Dirichlet problem for the Bellman equation:

\begin{equation}\label{bellman}
\left\{
\begin{array}{rcll}
\displaystyle\sup_{\alpha\in A}\big[L^\alpha v-c^\alpha v+f^\alpha\big]&=&0 &\text{in } D\\
v&=&g &\text{on }\partial D.
\end{array}
\right. 
\end{equation}

Define
\begin{equation}\label{condition}
\mu(x,\xi):=\inf_{\zeta: (\xi,\zeta)=1}\sup_{\alpha\in A}a^{ij}(\alpha,x)\zeta^i\zeta^j,
\end{equation}
\begin{equation}\label{conditions}
\mu(x):=\inf_{|\zeta|=1}\sup_{\alpha\in A}a^{ij}(\alpha,x)\zeta^i\zeta^j.
\end{equation}

The condition $\mu(x,\xi)>0$ means that $v_{(\xi)(\xi)}(x)$ is actually ``present" in the Bellman equation in (\ref{bellman}). More precisely, for any fixed $x\in D$ and $\xi\in \Rd\setminus\{ 0\}$, $\mu(x,\xi)>0$ if and only if there exists a control $\alpha\in A$ such that the corresponding diffusion matrix $a^\alpha(x)$ is non-degenerate in the direction $\xi$. For example, consider the linear equation
\begin{equation}\label{exa}
u_{x^1x^1}+2u_{x^1x^2}+u_{x^2x^2}=0.
\end{equation}
By (\ref{condition}), here
$$\mu(x,\xi)=\inf_{(\xi,\zeta)=1}(\zeta^1+\zeta^2)^2.$$
$\mu(x,\xi)>0$ if and only if $\xi\parallel\xi_0=(1,1)$. So only $u_{(\xi_0)(\xi_0)}$ is ``present'' in (\ref{exa}). In fact, the equation (\ref{exa}) can be rewritten as $$u_{(\xi_0)(\xi_0)}=0,$$ so that no other second-order derivatives is actually ``present'' in the equation, even though $u_{x^1x^1}$ and $u_{x^2x^2}$ exist explicitly in (\ref{exa}).

Also, it is not hard to see that
$$\mu(x)=\inf_{|\xi|=1}\mu(x,\xi).$$
Note that we have $\mu(x)>0$ at a point $x$ if and only if for any $\xi\ne0$, there exists a control $\alpha\in A$, such that the corresponding diffusion term $a^\alpha(x)$ is non-degenerate in the direct of $\xi$.

Let $\mathfrak{B}$ be the set of all skew-symmetric $d_1\times d_1$ matrices. For any positive constant $\lambda$, define
$$D_\lambda=\{x\in D: \psi(x)>\lambda\}.$$

\begin{assumption}\label{nondeg}
(uniform non-degeneracy along the normal to the boundary) There exists a positive constant $\delta_0$, such that
\begin{equation}
(a^\alpha n,n)\ge\delta_0 \mbox{ on }\partial D, \forall\alpha\in A,
\end{equation}
where $n$ is the unit normal vector.
\end{assumption}
\begin{assumption}\label{inequality}
(interior condition to control the moments of quasiderivatives, weaker than the non-degeneracy) There exist a function $\rho^\alpha(x): A\times D\rightarrow\Rd$, bounded on every set in the form of $A\times D_\lambda$ for all $\lambda>0$, a 
function $Q^\alpha(x,y): A\times D\times\Rd\rightarrow\mathfrak B$, bounded with respect to $(\alpha, x)$ on every set in the form of $A\times D_\lambda$ for all $\lambda>0, y\in\Rd$ and linear in $y$,  and a function $M^\alpha(x): A\times D\rightarrow\mathbb R$, bounded on every set in the form of $A\times D_\lambda$ for all $\lambda>0$, such that for any $\alpha\in A$, $x\in D$ and $|y|=1$,
\begin{equation}\label{ineq}
\begin{gathered}
\big\|\sigma^\alpha_{(y)}(x)+(\rho^\alpha(x),y)\sigma^\alpha(x)+\sigma^\alpha(x)Q^\alpha(x,y)\big\|^2+\\
\ 2\big(y,b^\alpha_{(y)}(x)+2(\rho^\alpha(x),y)b^\alpha(x)\big)\le c^\alpha(x)+M^\alpha(x)\big(a^\alpha(x)y,y\big).
\end{gathered}
\end{equation}

\end{assumption}

Our main result is the following:

\begin{theorem}\label{4c} 
Suppose that Assumptions \ref{nondeg} and \ref{inequality} hold. 
\begin{enumerate}
\item If for any $\alpha\in A$, $f^\alpha, g\in C^{0,1}(\bar{D})$, satisfying
$$\sup_{\alpha\in A}|f^\alpha|_{0,1,D}+|g|_{0,1,D}\le K_0,$$
then $v\in C^{0,1}(D)$, and for any $\xi\in \Rd$,
\begin{equation}
\big|v_{(\xi)}(x)\big|\le N\bigg(|\xi|+\frac{|\psi_{(\xi)}|}{\psi^{\frac{1}{2}}}\bigg), a.e. \mbox{ in }D,\label{4d}
\end{equation}
where the constant $N$ depends only on $d$, $d_1$ and $K_0$.
\item If for any $\alpha\in A$, $f^\alpha\in C^{0,1}(\bar{D}), g\in C^{1,1}(\bar{D})$, satisfying
$$\sup_{\alpha\in A}|f^\alpha|_{0,1,D}+|g|_{1,1,D}\le K_0,$$
and $f^\alpha+K_0|x|^2$ is convex,
then for any $\xi\in \Rd$, 
\begin{equation}
 v_{(\xi)(\xi)}(x)\ge -N\bigg(|\xi|^2+\frac{\psi_{(\xi)}^2}{\psi}\bigg), a.e. \mbox{ in }D,\label{4dd}
\end{equation}
\begin{equation}\label{4ddd}
v_{(\xi)(\xi)}(x)\le \mu(x,{\xi}/{|\xi|})^{-1}N\frac{|\xi|^2}{\psi}, a.e. \mbox{ in }D(\xi),
\end{equation}
where $D(\xi):=\{x\in D: \mu(x,\xi)>0\}$, and the constant $N$ depends only on $d$, $d_1$ and $K_0$. 

\item If $\mu(x)>0$ in $D$, then $v\in C_{loc}^{1,1}(D)$. In addition, $v$ given by (\ref{vax}) is the unique solution in $C_{loc}^{1,1}(D)\cap C^{0,1}(\bar D)$ of
\begin{equation}
\left\{
\begin{array}{rcll}
\displaystyle\sup_{\alpha\in A}\big[L^\alpha v(x)-c(\alpha, x)v(x)+f(\alpha,x)\big]&=&0 &\text{a.e. in } D\\
v&=&g &\text{on }\partial D.
\end{array}
\right.  \label{bellmanae}
\end{equation}
\end{enumerate}
\end{theorem}

We emphasize that the constants $N$ in (\ref{4d}), (\ref{4dd}) and (\ref{4ddd}) are independent of $\rho^\alpha, Q^\alpha$ and $M^\alpha$ in (\ref{ineq}).

\begin{remark}
The author doesn't know whether the estimates (\ref{4d}), (\ref{4dd}) and (\ref{4ddd}) are sharp.
\end{remark}

\begin{remark}
Refer to Remark 3.2 in \cite{quasi-linear} to see why Assumption \ref{inequality} is necessary under Assupmtion \ref{nondeg} and how to take advantage of the parameters $\rho^\alpha, Q^\alpha$ and $M^\alpha$ in (\ref{ineq}).
\end{remark}

\section{Auxiliary Convergence Results}

Let $U$ be a connected open subset in $\Rd$. Assume that, for any $\alpha\in \mathfrak A,\omega\in\Omega, t\ge0$, and $x\in U$, we are given a $d\times d_1$ matrix $\kappa_t^\alpha(x)$ and a $d$-dimensional vector $\nu_t^\alpha(x)$. We assume that $\kappa_t^\alpha$ and $\nu_t^\alpha$ are continuous in $x$ for any $\alpha,\omega,t$, measurable in $(\omega,t)$ for any $\alpha,x$, and $\mathcal{F}_t$-measurable in $\omega$ for any $\alpha,t,x$. 
Assume that for any $\alpha\in\mathfrak{A}$, the It\^o equation 
\begin{equation}\label{itozeta}
d\zeta^{\alpha,\zeta}_t=\kappa_t^{\alpha}(\zeta^{\alpha,\zeta}_t)dw_t+\nu_t^{\alpha}(\zeta^{\alpha,\zeta}_t)dt
\end{equation}
has a unique solution.

We suppose that for an $\epsilon_0\in(0,1]$ and for each $\epsilon\in[0,\epsilon_0]$, we are given
$$\kappa_t^\alpha(\epsilon)=\kappa_t^\alpha(x,\epsilon),\qquad\nu_t^\alpha(\epsilon)=\nu_t^\alpha(x,\epsilon)$$
having the same meaning and satisfying the same assumptions as those of $\kappa_t^\alpha$ and $\nu_t^\alpha$. Assume that for any $\alpha\in\mathfrak{A}$, the It\^o equation (\ref{itozeta}) corresponding to $\kappa_t^\alpha(\epsilon)$ and $\nu_t^\alpha(\epsilon)$ with initial condition $\zeta(\epsilon)\in U$ 
\begin{equation}\label{itozetae}
d\zeta^{\alpha,\zeta(\epsilon)}_t(\epsilon)=\kappa_t^\alpha(\zeta^{\alpha,\zeta(\epsilon)}_t(\epsilon),\epsilon)dw_t+\nu_t^\alpha(\zeta^{\alpha,\zeta(\epsilon)}_t(\epsilon),\epsilon)dt
\end{equation}
has a unique solution denoted by $\zeta^{\alpha,\zeta(\epsilon)}_t(\epsilon)$. 

\begin{lemma}\label{momentlemma} Let $q\in[2,\infty)$, $\theta\in(0,1)$, $M\in[0,\infty)$ be constants and $M^\alpha_t$ be a $\mathcal{F}_t$-adapted nonnegative process for any $\alpha\in \mathfrak{A}$.
\begin{enumerate}
\item If for any $\alpha\in \mathfrak A,t\ge0,x\in U$,
\begin{equation}\label{condone}
\|\kappa_t^\alpha(x)\|+|\nu_t^\alpha(x)|\le M|x|+M^\alpha_t,
\end{equation}
then for any bounded stopping times $\gamma^\alpha\le\tau^{\alpha,\zeta}_U$, $\forall \alpha$
\begin{equation}\label{ieone}
\begin{gathered}
\sup_{\alpha\in\mathfrak A}E^\alpha_\zeta\sup_{t\le\gamma}e^{-Nt}|\zeta_t|^q\\
\le |\zeta|^q+(2q-1)\sup_{\alpha\in\mathfrak A}E^\alpha\int_0^\gamma M_t^qe^{-Nt}dt,
\end{gathered}
\end{equation}
\begin{equation}\label{ietwo}
\begin{gathered}
\sup_{\alpha\in\mathfrak A}E^\alpha_\zeta\sup_{t\le\gamma}e^{-Nt}|\zeta_t|^{q\theta}\\
\le \frac{2-\theta}{1-\theta}\bigg(|\zeta|^{q\theta}+(2q-1)^\theta\sup_{\alpha\in\mathfrak A}E^\alpha\Big(\int_0^\gamma M_t^qe^{-Nt}dt\Big)^\theta\bigg),
\end{gathered}
\end{equation}
where $N=N(q,M)$ is a sufficiently large constant.
\item If for any $\alpha\in \mathfrak A,t\ge0,x\in U$, and some $\epsilon\in [0,\epsilon_0]$,
\begin{equation}\label{condtwo}
\|\kappa_t^\alpha(x)-\kappa_t^\alpha(y,\epsilon)\|+|\nu_t^\alpha(x)-\nu_t^\alpha(y,\epsilon)|\le M|x-y|+\epsilon M_t^\alpha,
\end{equation}
then for any bounded stopping times $\gamma^\alpha\le\tau^{\alpha,\zeta}_U\wedge\tau^{\alpha,\zeta(\epsilon)}_U(\epsilon)$, $\forall \alpha$
\begin{equation}\label{iethree}
\begin{gathered}
\sup_{\alpha\in\mathfrak A}E\sup_{t\le\gamma^\alpha}e^{-Nt}|\zeta_t^{\alpha,\zeta(\epsilon)}(\epsilon)-\zeta_t^{\alpha,\zeta}|^q\\
\le |\zeta(\epsilon)-\zeta|^q+\epsilon^q(2q-1)\sup_{\alpha\in\mathfrak A}E^\alpha\int_0^\gamma M_t^qe^{-Nt}dt,
\end{gathered}
\end{equation}
\begin{equation}\label{iefour}
\begin{gathered}
\sup_{\alpha\in\mathfrak A}E\sup_{t\le\gamma^\alpha}e^{-Nt}|\zeta_t^{\alpha,\zeta(\epsilon)}(\epsilon)-\zeta_t^{\alpha,\zeta}|^{q\theta}\\
\le\frac{2-\theta}{1-\theta}\bigg( |\zeta(\epsilon)-\zeta|^{q\theta}+\epsilon^{q\theta}(2q-1)^\theta\sup_{\alpha\in\mathfrak A}E^\alpha\Big(\int_0^\gamma M_t^qe^{-Nt}dt\Big)^\theta\bigg),
\end{gathered}
\end{equation}
where $N=N(q,M)$ is a sufficiently large constant.
\end{enumerate}
\end{lemma}

\begin{remark}
Observe that $q\theta$ covers $(0,\infty)$.
\end{remark}

\begin{proof}
It suffices to prove the uncontrolled version of (\ref{ieone}), (\ref{ietwo}), (\ref{iethree}) and (\ref{iefour}), so we drop the index $\alpha$ in what follows for simplicity of notation. We also abbreviate $\zeta_t^{\alpha,\zeta}$ to $\zeta_t$ and $\zeta_t^{\alpha,\zeta(\epsilon)}(\epsilon)$ to $\zeta_t(\epsilon)$. 

Also, choosing a localizing sequence of stopping times $\gamma_n\uparrow\infty$ such that $\int_0^{t\wedge\gamma_n}M_s^qe^{-Ns}ds$ are bounded for every n, we see, in view of the Monotone Convergence Theorem, that it will suffice to consider the case in which $\int_0^{t}M_s^qe^{-Ns}ds$ are bounded with respect to $(\omega,t)$.

By It\^o's formula, we have
\begin{align*}
de^{-Nt}|\zeta_t|^q=&e^{-Nt}\bigg[q|\zeta_t|^{q-2}(\zeta_t,\nu_t(\zeta_t))+\frac{q}{2}|\zeta_t|^{q-2}\|\kappa_t(\zeta_t)\|^2\\
&+\frac{q(q-2)}{2}|\zeta_t|^{q-4}|\kappa_t^*(\zeta_t)\zeta_t|^2-N|\zeta_t|^q\bigg]dt+dm_t,
\end{align*}
where $m_t$ is a local martingale starting from zero. From (\ref{condone}) we have,
\begin{equation}\label{39}
\|\kappa_t(\zeta_t)\|+|\nu_t(\zeta_t)|\le M|\zeta_t|+M_t.
\end{equation}
By Young's inequality
\begin{align*}
q|\zeta_t|^{q-2}(\zeta_t,\nu_t(\zeta_t))&\le(qM+q-1)|\zeta_t|^q+M_t^q\\
\frac{q}{2}|\zeta_t|^{q-2}\|\kappa_t(\zeta_t)\|^2&\le q|\zeta_t|^{q-2}(M^2|\zeta_t|^2+M_t^2)\le(qM^2+q-2)|\zeta_t|^q+2M_t^q\\
\frac{q(q-2)}{2}|\zeta_t|^{q-4}|\kappa_t^*(\zeta_t)\zeta_t|^2&\le(q-2)\big[(qM^2+q-2)|\zeta_t|^q+2M_t^q\big]
\end{align*}
So for sufficiently large constant $N=N(q,M)$, we have
$$e^{-Nt}|\zeta_t|^q\le|\zeta|^q+(2q-1)\int_0^tM_t^qe^{-Nt}dt.$$
which implies that
$$E\sup_{t\le\gamma}e^{-Nt}|\zeta_t|^q\le|\zeta|^q+(2q-1)E\int_0^\gamma M_t^qe^{-Nt}dt.$$
Due to Lemma 7.3(ii) in \cite{MR1944761}, we conclude that
\begin{align*}
E\sup_{t\le\gamma}e^{-Nt}|\zeta_t|^{q\theta}\le& \frac{2-\theta}{1-\theta}E\bigg(|\zeta|^q+(2q-1)\int_0^\gamma M_t^qe^{-Nt}dt\bigg)^\theta\\
\le&\frac{2-\theta}{1-\theta}\bigg(|\zeta|^{q\theta}+(2q-1)^\theta E\Big(\int_0^\gamma M_t^qe^{-Nt}dt\Big)^\theta\bigg).
\end{align*}
Similarly, by It\^o's formula,
\begin{align*}
&d\Big(e^{-Nt}|\zeta_t(\epsilon)-\zeta_t|^q\Big)\\
=&e^{-Nt}\bigg[q|\zeta_t(\epsilon)-\zeta_t|^{q-2}\Big(\zeta_t(\epsilon)-\zeta_t,\nu_t(\zeta_t(\epsilon),\epsilon)-\nu_t(\zeta_t)\Big)\\
&+\frac{q}{2}|\zeta_t(\epsilon)-\zeta_t|^{q-2}\|\kappa_t(\zeta_t(\epsilon),\epsilon)-\kappa_t(\zeta_t)\|^2\\
&+\frac{q(q-2)}{2}|\zeta_t(\epsilon)-\zeta_t|^{q-4}\Big|(\kappa_t^*(\zeta_t(\epsilon),\epsilon)-\kappa_t^*(\zeta_t))(\zeta_t(\epsilon)-\zeta_t)\Big|^2\\
&-N|\zeta_t(\epsilon)-\zeta_t|^q\bigg]dt+dm_t,
\end{align*}
where $m_t$ is a local martingale starting at zero. By (\ref{condtwo}), we have
$$\|\kappa_t(\zeta_t(\epsilon),\epsilon)-\kappa_t(\zeta_t)\|+|\nu_t(\zeta_t(\epsilon),\epsilon)-\nu_t(\zeta_t)|\le M|\zeta_t(\epsilon)-\zeta_t|+\epsilon M_t,$$
which can play the same role as (\ref{39}). So (\ref{iethree}) and (\ref{iefour}) can be proved by mimicking the argument for proving (\ref{ieone}) and (\ref{ietwo}).

\end{proof}

Next, we introduce the quasiderivatives to be used in the proof of the main theorem and apply Lemmas \ref{momentlemma} to estimate moments of these quasiderivatives. 

For any $\alpha\in\mathfrak A$, let $r_t^\alpha,\hat r_t^\alpha,\pi_t^\alpha,\hat\pi_t^\alpha,P_t^\alpha,\hat P_t^\alpha$ be jointly measurable adapted processes with values in $\mathbb R$, $\mathbb R$, $\mathbb R^{d_1}$, $\mathbb R^{d_1}$, $\mathrm{Skew}(d_1,\mathbb R)$, $\mathrm{Skew}(d_1,\mathbb R)$, respectively, where $\mathrm{Skew}(d_1,\mathbb R)$ denotes the set of all $d_1\times d_1$ skew-symmetric real matrices. Let $\epsilon$ be a small positive constant. For each $\alpha\in\mathfrak A$, $x,y,z\in D$, $\xi,\eta\in \Rd$, we consider the It\^o equation (\ref{itox}) and the following four other It\^o equations:
\allowdisplaybreaks\begin{align}
\label{itoy}
dy_t^{\alpha,y}( \epsilon)=&\sqrt{1+2 \epsilon r_t^{\alpha}}\sigma^{\alpha_t}(y_t^{\alpha,y}( \epsilon))e^{ \epsilon P_t^{\alpha}}dw_t\\
\nonumber+&\Big[(1+2 \epsilon r_t^{\alpha})b^{\alpha_t}(y_t^{\alpha,y}( \epsilon))-\sqrt{1+2 \epsilon r_t^{\alpha}}\sigma^{\alpha_t}(y_t^{\alpha,y}( \epsilon))e^{ \epsilon P_t^{\alpha}} \epsilon\pi_t^{\alpha}\Big]dt,\\
\label{itoz}
dz_t^{\alpha,z}(\epsilon)=&\sqrt{1+2 \epsilon r_t^{\alpha}+ \epsilon^2\hat r_t^{\alpha}}\sigma^{\alpha_t}(z_t^{\alpha,z}(\epsilon))e^{ \epsilon P_t^{\alpha}}e^{\frac{\epsilon^2}{2}  \hat P_t^{\alpha}}dw_t\\
\nonumber+&\Big[(1+2 \epsilon r_t^{\alpha}+ \epsilon^2\hat r_t^{\alpha})b^{\alpha_t}(z_t^{\alpha,z}(\epsilon))\\
\nonumber&-\sqrt{1+2 \epsilon r_t^{\alpha}+ \epsilon^2\hat r_t^{\alpha}}\sigma^{\alpha_t}(z_t^{\alpha,z}(\epsilon))e^{ \epsilon P_t^{\alpha}}e^{\frac{\epsilon^2}{2} \hat P_t^{\alpha}}( \epsilon\pi_t^{\alpha}+\frac{\epsilon^2}{2} \hat\pi_t^{\alpha})\Big]dt,\\
\label{itoxi}
d\xi_t^{\alpha,\xi}=&\Big[\sigma^{\alpha_t}_{(\xi_t^{\alpha,\xi})}+r^\alpha_t\sigma^{\alpha_t}+\sigma^{\alpha_t} P^\alpha_t\Big]dw_t\\
\nonumber+ &\Big[b^{\alpha_t}_{(\xi_t^{\alpha,\xi})}+2r^\alpha_tb^{\alpha_t}-\sigma^{\alpha_t} \pi^\alpha_t\Big]dt,\\
\label{itoeta}
d\eta_t^{\alpha,\eta}=&\Big[\sigma^{\alpha_t}_{(\eta_t^{\alpha,\eta})}+\hat{r}^\alpha_t\sigma^{\alpha_t}+\sigma^{\alpha_t}\hat{P}^\alpha_t +\sigma^{\alpha_t}_{(\xi_t^{\alpha,\xi})(\xi_t^{\alpha,\xi})}+  2r^\alpha_t\sigma^{\alpha_t}_{(\xi_t^{\alpha,\xi})}\\
\nonumber&+2\sigma^{\alpha_t}_{(\xi_t^{\alpha,\xi})} P^\alpha_t+2r^\alpha_t\sigma^{\alpha_t} P^\alpha_t-(r^\alpha_t)^2\sigma^{\alpha_t}+\sigma^{\alpha_t} (P^\alpha_t)^2\Big]dw_t\\
\nonumber+  &\Big[b^{\alpha_t}_{(\eta_t^{\alpha,\eta})}+2\hat{r}^\alpha_t b^{\alpha_t}-\sigma^{\alpha_t}\hat{\pi}^\alpha_t+b^{\alpha_t}_{(\xi_t^{\alpha,\xi})(\xi_t^{\alpha,\xi})}+4r^\alpha_t b^{\alpha_t}_{(\xi_t^{\alpha,\xi})}\\
\nonumber&-2\sigma^{\alpha_t}_{(\xi_t^{\alpha,\xi})}\pi^\alpha_t-2r^\alpha_t\sigma^{\alpha_t}\pi^\alpha_t-2\sigma^{\alpha_t} P^\alpha_t \pi^\alpha_t\Big]dt,
\end{align}
where $\sigma^{\alpha}$ and $b^{\alpha}$ satisfy (\ref{tiao}) and we drop the arguments $x_t^{\alpha,x}$ in $\sigma^{\alpha_t}$ and $b^{\alpha_t}$ and their derivatives in (\ref{itoxi}) and (\ref{itoeta}). 

Let $\bar\tau_D^{\alpha,y}(\epsilon)$ be the first exit time of $y_t^{\alpha,y}(\epsilon)$ from $D$, and $\hat\tau_D^{\alpha,z}(\epsilon)$ be the first exit time of $z_t^{\alpha,z}(\epsilon)$ from $D$.

By Theorem 3.2.1 in \cite{MR2144644} we know that if
\begin{equation}\label{rpip1}
\begin{gathered}
\int_0^T(|r^\alpha_t|^2+|\pi^\alpha_t|^2+|P^\alpha_t|^2)dt<\infty, \\
\forall T\in[0,\infty), \forall\alpha\in\mathfrak A,
\end{gathered}
\end{equation}
then (\ref{itoy}) and (\ref{itoxi}) have unique solutions on $[0,\bar\tau_D^{\alpha,y}(\epsilon))$ and $[0,\tau_D^{\alpha,x})$, respectively. 

Similarly, it is shown in Theorem 2.1 in \cite{quasi-linear} that if
\begin{equation}\label{rpip2}
\begin{gathered}
\int_0^T(|\hat r^\alpha_t|^2+|\hat\pi^\alpha_t|^2+|\hat P^\alpha_t|^2+|r^\alpha_t|^4+|\pi^\alpha_t|^4+|P^\alpha_t|^4)dt<\infty,\\
 \forall T\in[0,\infty), \forall\alpha\in\mathfrak A,
\end{gathered}
\end{equation}
then (\ref{itoz}) and (\ref{itoeta}) have unique solutions on $[0,\hat\tau_D^{\alpha,z}(\epsilon))$ and $[0,\tau_D^{\alpha,x})$, respectively. 

In (\ref{itoy}) and (\ref{itoz}), notice that when $\epsilon=0$, we have $y_t^{\alpha,y}( 0)$ and $z_t^{\alpha,z}(0)$, which are nothing but $x_t^{\alpha,y}$ and $x_t^{\alpha,z}$. Therefore, $y_t^{\alpha,y}(\epsilon)$ and $z_t^{\alpha,z}(\epsilon)$ are perturbations of $x_t^{\alpha,x}$. In Theorems \ref{thm1} and \ref{thm2} we will prove that under suitable conditions, $\xi_t^{\alpha,\xi}$ and $\eta_t^{\alpha,\eta}$, given by (\ref{itoxi}) and (\ref{itoeta}), respectively,  are the first derivative of $y_t^{\alpha,x+\epsilon\xi}(\epsilon)$ and the second derivative of $z_t^{\alpha, x+\epsilon\xi+\epsilon^2\eta/2}(\epsilon)$ in some sense (see (\ref{res1b2}) and (\ref{res2b3})), respectively.

The auxiliary processes $r_t^\alpha$ and $\hat r_t^\alpha$ come from random time change. The processes $\pi_t^\alpha$ and $\hat\pi_t^\alpha$ are due to Girsanov's theorem on changing the probability space, and the processes $P_t^\alpha$ and $\hat P_t^\alpha$ are based on changing the Wiener process based on Levy's theorem. As discussed in Section 2 of \cite{quasi-linear}, thanks to the presence of these auxiliary processes, the quasiderivatives $\xi_t^{\alpha, \xi}$ and $\eta_t^{\alpha,\eta}$ enjoy certain freedom. It turns out that, heuristically, we can steer the quasiderivatives so that they are tangent to the boundary when $x_t^{\alpha,x}$ hit it. As a result, the directional derivatives of $v$ along the quasiderivatives become the derivatives of the boundary data $g$, and estimating the derivatives of $v$ is reduced to estimating the moments of the quasiderivatives.

\begin{theorem}\label{thm1} Given constants $p\in (0,\infty)$, $p'\in[0,p)$, $T\in [1,\infty)$, $x\in D$, $\xi\in\Rd$. Suppose (\ref{rpip1}) is satisfied. Assume that there exists a constant $K\in[1,\infty)$ and for any $\alpha\in\mathfrak A$, an adapted nonnegative process $K_t^\alpha$, such that
\begin{equation}\label{cond1}
|r_t^\alpha|+|\pi_t^\alpha|+|P_t^\alpha|\le K|\xi_t^{\alpha,\xi}|+K_t^\alpha, \forall\alpha.
\end{equation}
\begin{enumerate}
\item Given stopping times $\gamma^\alpha\le\tau_D^{\alpha,x}$, $\alpha\in\mathfrak A$, if
\begin{equation}\label{cond1a}
\sup_{\alpha\in\mathfrak A}E^\alpha\int_0^{\gamma\wedge T}K_t^{2\vee p}dt<\infty,
\end{equation}
then we have
\begin{equation}\label{res1a}
\sup_{\alpha\in\mathfrak A}E^\alpha_\xi\sup_{t\le\gamma\wedge T}|\xi_t|^p<\infty.
\end{equation}
\item Let the constant $\epsilon_0$ be sufficiently small so that $B(x,\epsilon_0|\xi|)\subset D$. For any $\epsilon\in[0,\epsilon_0]$, given stopping times $\gamma^\alpha(\epsilon)\le\tau_D^{\alpha,x}\wedge\bar\tau_D^{\alpha,x+\epsilon\xi}(\epsilon)$, $\alpha\in\mathfrak A$, if
\begin{equation}\label{cond1b}
\sup_{\epsilon\in[0,\epsilon_0]}\sup_{\alpha\in\mathfrak A}E^\alpha\int_0^{\gamma(\epsilon)\wedge T}K_t^{2(2\vee p)}dt<\infty,
\end{equation}
then we have
\begin{equation}\label{res1b1}
\lim_{\epsilon\downarrow0}\sup_{\alpha\in\mathfrak A}E\sup_{t\le\gamma^\alpha(\epsilon)\wedge T}\frac{|y_t^{\alpha,x+\epsilon\xi}(\epsilon)-x_t^{\alpha,x}|^p}{\epsilon^{p'}}=0,
\end{equation}
\begin{equation}\label{res1b2}
\lim_{\epsilon\downarrow0}\sup_{\alpha\in\mathfrak A}E\sup_{t\le\gamma^\alpha(\epsilon)\wedge T}|\frac{y_t^{\alpha,x+\epsilon\xi}(\epsilon)-x_t^{\alpha,x}}{\epsilon}-\xi_t^{\alpha,\xi}|^{p/2}=0.
\end{equation}
\end{enumerate}
\end{theorem}
\begin{proof}
In the proof, we drop the superscripts $\alpha$, $\alpha_t$, etc., when this will not cause confusion.

To prove (1) we consider the It\^o equation (\ref{itozeta}) in which $\zeta_t^{\alpha,\zeta}=\xi_t^{\alpha,\xi}$. By conditions (\ref{tiao}) and (\ref{cond1}), we have
$$\big\|\sigma_{(\xi_t)}+r_t\sigma+\sigma P_t\big\|+\big|b_{(\xi_t^{\alpha,\xi})}+2r_tb-\sigma \pi_t\big|\le M|\xi_t|+M_t, \forall\alpha,$$
where $M=N(K,K_0), M_t^\alpha=N(K_0)K_t^\alpha$.
Applying Lemma \ref{momentlemma}(1), we have
\begin{equation*}
\sup_{\alpha\in\mathfrak A}E^\alpha_\xi\sup_{t\le\gamma\wedge T}|\xi_t|^p\le\left\{
\begin{array}{l}
e^{NT}(|\xi|^p+(2p-1)\sup_{\alpha\in\mathfrak A}E^\alpha\int_0^{\gamma\wedge T}M_t^pdt)\mbox{ if }p\ge2\\
e^{NT}\frac{4-p}{2-p}(|\xi|^p+3^{\frac{p}{2}}(\sup_{\alpha\in\mathfrak A}E^\alpha\int_0^{\gamma\wedge T}M_t^2dt)^{\frac{p}{2}})\mbox{ if }p<2.
\end{array}
\right. 
\end{equation*}
To prove (2) we first consider the It\^o equations (\ref{itozeta}) and (\ref{itozetae}) in which 
$$\zeta_t^{\alpha,\zeta}=x_t^{\alpha,x},\qquad\zeta_t^{\alpha,\zeta(\epsilon)}(\epsilon)=y_t^{\alpha,x+\epsilon\xi}(\epsilon).$$
Notice that
\begin{align*}
\|\kappa_t(y,\epsilon)-\kappa_t(x)\|=&\|\sqrt{1+2 \epsilon r_t}\sigma(y)e^{ \epsilon P_t}-\sigma(x)\|\\
\le&|\sqrt{1+2 \epsilon r_t}-1|\|\sigma(y)e^{ \epsilon P_t}\|+\|\sigma(y)\| \| e^{\epsilon P_t} -I_{d_1\times d_1}\|\\
&+\|\sigma(y)-\sigma(x)\|\\
\le&2\epsilon |r_t| K_0+K_0\epsilon e^{\epsilon' P_t}+K_0|y-x|\\
\le& M|y-x|+\epsilon M_t,
\end{align*}
where $\epsilon'\in[0,\epsilon]$ is due to Taylor's theorem with Lagrange remainder. Similarly,
\begin{align*}
|\nu_t(y,\epsilon)-\nu_t(x)|=&|(1+2 \epsilon r_t)b(y)-\sqrt{1+2 \epsilon r_t}\sigma(y)e^{ \epsilon P_t} \epsilon\pi_t-b(x)|\\
\le&2\epsilon|r_t|K_0+(1+\epsilon|r_t|)K_0\epsilon|\pi_t|+K_0|y-x|\\
\le&M|y-x|+\epsilon M_t,
\end{align*}
where $M=K_0, M_t^\alpha=N(K,K_0)(|\xi_t^{\alpha,\xi}|^2+(K_t^\alpha)^2\vee 1)$.
Applying Lemma \ref{momentlemma}(2), we have
$$\sup_{\alpha\in\mathfrak A}E\sup_{t\le\gamma^\alpha(\epsilon)\wedge T}|y_t^{\alpha,x+\epsilon\xi}(\epsilon)-x_t^{\alpha,x}|^p$$
\begin{equation*}
\le\left\{
\begin{array}{l}
\epsilon^pe^{NT}(|\xi|^p+(2p-1)\sup_{\alpha\in\mathfrak A}E^\alpha\int_0^{\gamma(\epsilon)\wedge T}M_t^pdt)\mbox{ if }p\ge2\\
\epsilon^pe^{NT}\frac{4-p}{2-p}(|\xi|^p+3^{\frac{p}{2}}(\sup_{\alpha\in\mathfrak A}E^\alpha\int_0^{\gamma(\epsilon)\wedge T}M_t^2dt)^{\frac{p}{2}})\mbox{ if }p<2.
\end{array}
\right. 
\end{equation*}
Due to (\ref{cond1b}) and (\ref{res1a}), we have
$$\sup_{[0,\epsilon_0]}\sup_{\alpha\in\mathfrak A}E^\alpha\int_0^{\gamma(\epsilon)\wedge T}M_t^{2\vee p}dt<\infty,$$
which completes the proof of (\ref{res1b1}).

Next, we first consider the It\^o equations (\ref{itozeta}) and (\ref{itozetae}) in which 
$$\zeta_t^{\alpha,\zeta}=\xi_t^{\alpha,\xi},\qquad\zeta_t^{\alpha,\zeta(\epsilon)}(\epsilon)=\xi_t^{\alpha,\xi}(\epsilon):=\frac{y_t^{\alpha,x+\epsilon\xi}(\epsilon)-x_t^{\alpha,x}}{\epsilon}.$$
Observe that, by mean value theorem
\begin{align*}
\|\frac{\sigma(y_t(\epsilon))-\sigma(x_t)}{\epsilon}-\sigma_{(\xi_t)}(x_t)\|=&\|\sigma_{(\xi_t(\epsilon))}(y^*_t(\epsilon))-\sigma_{(\xi_t)}(x_t)\|\\
=&\|\sigma_{(\xi_t(\epsilon))}(y^*_t(\epsilon))-\sigma_{(\xi_t(\epsilon))}(x_t)\|+\|\sigma_{(\xi_t(\epsilon))}(x_t)-\sigma_{(\xi_t)}(x_t)\|\\
\le &|\xi_t(\epsilon)|\|\sigma_x(y_t^*(\epsilon))-\sigma_x(x_t)\|I_{|y_t(\epsilon)-x_t|\le\delta}\\
&+|\xi_t(\epsilon)|I_{|y_t(\epsilon)-x_t|>\delta}+K_0|\xi_t(\epsilon)-\xi_t|,\\
|\frac{b(y_t(\epsilon))-b(x_t)}{\epsilon}-b_{(\xi_t)}(x_t)|\le &|\xi_t(\epsilon)|\|b_x(y_t^*(\epsilon))-b_x(x_t)\|I_{|y_t(\epsilon)-x_t|\le\delta}\\
&+|\xi_t(\epsilon)|I_{|y_t(\epsilon)-x_t|>\delta}+K_0|\xi_t(\epsilon)-\xi_t|,\\
|\frac{\sqrt{1+2\epsilon r_t}-1}{\epsilon}-r_t|=&|r_t(\frac{2}{\sqrt{1+2\epsilon r_t}+1}-1)|\\
=&|\frac{-2\epsilon r_t^2}{(1+\sqrt{1+2\epsilon r_t})^2}|\le2\epsilon |r_t|^2,\\
\|\frac{e^{\epsilon P_t}-1}{\epsilon}-P_t\|=&\frac{\epsilon}{2}\|P_t^2e^{\epsilon'P_t}\|\le\frac{\epsilon}{2}\|P_t\|^2.
\end{align*}
The equation (\ref{res1b2}) can be proved by mimicking the proof of (\ref{res1b1}).
\end{proof}
\begin{theorem}\label{thm2}
Given constants $p\in (0,\infty)$, $p'\in[0,p)$, $T\in [1,\infty)$, $x\in D$, $\xi\in\Rd$, $\eta\in\Rd$. Suppose (\ref{rpip2}) is satisfied. Assume that there exists a constant $K\in[1,\infty)$ and for any $\alpha\in\mathfrak A$, an adapted nonnegative process $K_t^\alpha$, such that 
\begin{equation}\label{cond2}
|\hat r_t^\alpha|+|\hat \pi_t^\alpha|+|\hat P_t^\alpha|+|r_t^\alpha|^2+|\pi^\alpha_t|^2+|P_t^\alpha|^2
\le K(|\eta_t^{\alpha,\eta}|+|\xi_t^{\alpha,\xi}|^2)+K_t^\alpha, \forall\alpha.
\end{equation}
\begin{enumerate}
\item Given stopping times $\gamma^\alpha\le\tau_D^{\alpha,x}$, $\alpha\in\mathfrak A$, if (\ref{cond1a}) holds, then we have (\ref{res1a}) and
\begin{equation}\label{res2a}
\sup_{\alpha\in\mathfrak A}E^\alpha_\eta\sup_{t\le\gamma\wedge T}|\eta_t|^p<\infty.
\end{equation}
\item Let the constant $\epsilon_0$ be sufficiently small so that $B(x,\epsilon_0|\xi|)\subset D$. For any $\epsilon\in[0,\epsilon_0]$, let 
$$x(\epsilon)=x+\epsilon\xi+\frac{\epsilon^2}{2}\eta.$$ 
If (\ref{cond1b}) holds for given stopping times $\gamma_2^\alpha(\epsilon)$ satisfying
$$\gamma_2^\alpha(\epsilon)\le\tau_D^{\alpha,x}\wedge\hat\tau_D^{\alpha,x(\epsilon)}(\epsilon), \alpha\in\mathfrak A,$$ 
then we have
\begin{equation}\label{res2b1}
\lim_{\epsilon\downarrow0}\sup_{\alpha\in\mathfrak A}E\sup_{t\le\gamma_2^\alpha(\epsilon)\wedge T}\frac{|z_t^{\alpha,x(\epsilon)}(\epsilon)-x_t^{\alpha,x}|^{p}}{\epsilon^{p'}}=0,
\end{equation}
\begin{equation}\label{res2b1'}
\lim_{\epsilon\downarrow0}\sup_{\alpha\in\mathfrak A}E\sup_{t\le\gamma_2^\alpha(-\epsilon)\wedge T}\frac{|z_t^{\alpha,x(-\epsilon)}(-\epsilon)-x_t^{\alpha,x}|^{p}}{\epsilon^{p'}}=0,
\end{equation}
\begin{equation}\label{res2b2}
\lim_{\epsilon\downarrow0}\sup_{\alpha\in\mathfrak A}E\sup_{t\le\gamma_2^\alpha(\epsilon)\wedge T}|\frac{z_t^{\alpha,x(\epsilon)}(\epsilon)-x_t^{\alpha,x}}{\epsilon}-\xi_t^{\alpha,\xi}|^p=0.
\end{equation}
If (\ref{cond1b}) holds for given stopping times $\gamma_3^\alpha(\epsilon)$ satisfying
$$\gamma_3^\alpha(\epsilon)\le\tau_D^{\alpha,x}\wedge\hat\tau_D^{\alpha,x(\epsilon)}(\epsilon)\wedge\hat\tau_D^{\alpha,x(-\epsilon)}(-\epsilon), \alpha\in\mathfrak A,$$ 
then we have
\begin{equation}\label{res2b3}
\lim_{\epsilon\downarrow0}\sup_{\alpha\in\mathfrak A}E\sup_{t\le\gamma_3^\alpha(\epsilon)\wedge T}\bigg|\frac{z_t^{\alpha,x(\epsilon)}(\epsilon)-2x_t^{\alpha,x}+z_t^{\alpha,x(-\epsilon)}(-\epsilon)}{\epsilon^2}-\eta_t^{\alpha,\eta}\bigg|^{p/2}=0.
\end{equation}
\end{enumerate}
\end{theorem}

\begin{proof}
Again, we drop superscripts $\alpha$, $\alpha_t$, etc., when this will cause no confusion.

The inequality (\ref{res2a}) can be proved by observing that (\ref{cond2}) and (\ref{cond1a}) imply that
$$\sup_{\alpha\in\mathfrak A}E^\alpha_\xi\sup_{t\le\gamma\wedge T}|\xi_t|^{2p}<\infty$$
and then mimicking the proof of (\ref{res1a}).

The equations (\ref{res2b1}) and (\ref{res2b2}) are obtained by repeating the proof of (\ref{res1b1}) and (\ref{res1b2}). The equation (\ref{res2b1'}) is obvious once we get (\ref{res2b1}).

To proof (\ref{res2b3}), we observe that, for example, 
\begin{align*}
&\frac{\sigma(z_t(\epsilon))-2\sigma(x_t)+\sigma(z_t(-\epsilon))}{\epsilon^2}\\
=&\frac{1}{\epsilon^2}[\sigma_{(z_t(\epsilon)-x_t)}(x_t)+\frac{1}{2}\sigma_{(z_t(\epsilon)-x_t)(z_t(\epsilon)-x_t)}( z^*_t(\epsilon))\\
&+\sigma_{(z_t(-\epsilon)-x_t)}(x_t)+\frac{1}{2}\sigma_{(z_t(-\epsilon)-x_t)(z_t(-\epsilon)-x_t)}(z^*_t(-\epsilon))]\\
=&\sigma_{(\eta_t(\epsilon))}+\frac{1}{2}[\sigma_{(\xi_t(\epsilon))(\xi_t(\epsilon))}( z^*_t(\epsilon))+\sigma_{(\xi_t(-\epsilon))(\xi_t(-\epsilon))}( z^*_t(-\epsilon))],
\end{align*}
where
$$\eta_t(\epsilon)=\frac{z_t(\epsilon)-2x_t+z_t(-\epsilon)}{\epsilon^2},\qquad\xi_t(\epsilon)=\frac{z_t(\epsilon)-x_t}{\epsilon},$$
$z^*_t(\epsilon)$ is a point on the straight line segment with endpoints $x_t$ and $z_t(\epsilon)$, and $z^*_t(-\epsilon)$ is a point on the straight line segment with endpoints $x_t$ and $z_t(-\epsilon)$.

It follows that
\begin{align*}
&\|\frac{\sigma(z_t(\epsilon))-2\sigma(x_t)+\sigma(z_t(-\epsilon))}{\epsilon^2}-\sigma_{(\eta_t)}(x_t)-\sigma_{(\xi_t)(\xi_t)}(x_t)\|\\
\le&K_0|\frac{z_t(\epsilon)-2x_t+z_t(-\epsilon)}{\epsilon^2}-\eta_t|\\
&+\frac{1}{2}|\frac{z_t(\epsilon)-x_t}{\epsilon}|^2\Big(\|\sigma_{xx}( z^*_t(\epsilon))-\sigma_{xx}(x_t)\|+\|\sigma_{xx}( z^*_t(-\epsilon))-\sigma_{xx}(x_t)\|\Big)I_{|z_t(\epsilon)-x_t|\le\delta}\\
&+\|\sigma\|_{2,D}|\frac{z_t(\epsilon)-x_t}{\epsilon}|^2\Big(I_{|z_t(\epsilon)-x_t|>\delta}+I_{|z_t(-\epsilon)-x_t|>\delta}\Big)\\
&+K_0|\frac{z_t(\epsilon)-x_t}{\epsilon}-\xi_t|^2+K_0|\frac{z_t(-\epsilon)-x_t}{\epsilon}-\xi_t|^2.
\end{align*}
It remains to mimic the proof of (\ref{res1b2}).
\end{proof}

We end up this section by showing a convergence result about the stopping times which will be applied in the next section.

\begin{theorem}
Let $\delta$ be a positive constant such that $D_\delta=\{x\in D:\psi>\delta\}$ is nonempty, and $\delta_1,\delta_2$ be positive constants satisfying $\delta_1<\delta_2$. Let $D_{\delta_1}^{\delta_2}=\{x\in D:\delta_1<\psi<\delta_2\}$. Then for any $x\in D$, if (\ref{res1b1}) holds with 
$$\gamma^\alpha(\epsilon)=\tau_D^{\alpha,x}\wedge\bar \tau_D^{\alpha,x+\epsilon\xi}(\epsilon),$$
 for $p=1,p'=0$ and $\forall T\in [1,\infty)$, then we have
\begin{equation}\label{stopping1}
\lim_{\epsilon\downarrow0}\sup_{\alpha\in\mathfrak A}E(\tau^{\alpha,x}_D-\tau^{\alpha,x}_D\wedge\bar\tau^{\alpha,x+\epsilon\xi}_D(\epsilon))=0.
\end{equation}
For any $x\in D$, if (\ref{res2b1}) and (\ref{res2b1'}) hold with 
$$\gamma^\alpha(\epsilon)=\tau_D^{\alpha,x}\wedge\hat \tau_D^{\alpha,x(\epsilon)}(\epsilon) \mbox{ and }\gamma^\alpha(-\epsilon)=\tau_D^{\alpha,x}\wedge\hat \tau_D^{\alpha,x(-\epsilon)}(-\epsilon),$$
 respectively, for $p=1,p'=0$ and $\forall T\in [1,\infty)$, then we have
\begin{equation}\label{stopping2}
\lim_{\epsilon\downarrow0}\sup_{\alpha\in\mathfrak A}E(\tau^{\alpha,x}_D-\tau^{\alpha,x}_D\wedge\hat\tau^{\alpha,x(\epsilon)}_D(\epsilon)\wedge\hat\tau^{\alpha,x(-\epsilon)}_D(\epsilon))=0.
\end{equation}
The statement still holds when replacing $D$ by $D_\delta$ or $D_{\delta_1}^{\delta_2}$, provided that $\delta_2$ is sufficiently small.
\end{theorem}
\begin{proof}
We drop the subscript $D$ and the argument $\epsilon$ for simplicity of notation. Notice that, for any $\alpha\in\mathfrak A$,
\allowdisplaybreaks\begin{align*}
E\big(\tau^{\alpha,x}-\gamma^\alpha\big)=&E\int_{\gamma^\alpha}^{\tau^{\alpha,x}}1dt\\
\le&-E\int_{\gamma^\alpha}^{\tau^{\alpha,x}}L^\alpha\psi(x_t^{\alpha,x})dt\\
=&-E\Big(\psi\big(x_{\tau^{\alpha,x}}^{\alpha,x}\big)-\psi\big(x_{\gamma^\alpha}^{\alpha,x}\big)\Big)I_{\gamma^\alpha<\tau^{\alpha,x}}\\
=&E\psi\big(x_{\bar\tau^{\alpha,x+ \epsilon\xi}}^{\alpha,x}\big)
I_{\bar\tau^{\alpha,x+ \epsilon\xi}<\tau^{\alpha,x}}\\
=&E\Big(\psi\big(x_{\bar\tau^{\alpha,x+ \epsilon\xi}}^{\alpha,x}\big)-\psi\big(y_{\bar\tau^{\alpha,x+ \epsilon\xi}}^{\alpha,x+ \epsilon\xi}\big)\Big)
I_{\bar\tau^{\alpha,x+ \epsilon\xi}<\tau^{\alpha,x}}\\
\le&E\Big(\psi\big(x_{\bar\tau^{\alpha,x+ \epsilon\xi}}^{\alpha,x}\big)-\psi\big(y_{\bar\tau^{\alpha,x+ \epsilon\xi}}^{\alpha,x+ \epsilon\xi}\big)\Big)
I_{\bar\tau^{\alpha,x+ \epsilon\xi}<\tau^{\alpha,x}\le T}+2K_0P^\alpha_x(\tau>T).
\end{align*}
Due to (\ref{res1b1}), we have
\begin{align*}
&\varlimsup_{ \epsilon\downarrow0}\Big(\sup_\alpha E\Big(\psi\big(x_{\bar\tau^{\alpha,x+ \epsilon\xi}}^{\alpha,x}\big)-\psi\big(y_{\bar\tau^{\alpha,x+ \epsilon\xi}}^{\alpha,x+ \epsilon\xi}\big)\Big)
I_{\bar\tau^{\alpha,x+ \epsilon\xi}<\tau^{\alpha,x}\le T}\Big)\\
\le&\sup_D|\psi_x|\cdot\varlimsup_{ \epsilon\downarrow0}\Big(\sup_\alpha E\Big|x_{\bar\tau^{\alpha,x+ \epsilon\xi}}^{\alpha,x}-y_{\bar\tau^{\alpha,x+ \epsilon\xi}}^{\alpha,x+ \epsilon\xi}\Big|
I_{\bar\tau^{\alpha,x+ \epsilon\xi}<\tau^{\alpha,x}\le T}\Big)\\
=&0.
\end{align*}
Also, notice that for any $\alpha\in\mathfrak A, T\in[1,\infty)$,
$$P^\alpha_x(\tau>T)\le\frac{1}{T}E_x^\alpha\tau\le\frac{1}{T}E_x^\alpha\int_0^{\tau}\Big(-L^\alpha\psi(x_t)\Big)dt=\frac{1}{T}\Big(\psi(x)-\psi(x_{\tau^{\alpha,x}}^{\alpha,x})\Big)\le\frac{K_0}{T}.$$
It turns out that
$$\varlimsup_{\epsilon\downarrow0}\sup_{\alpha\in\mathfrak A}E(\tau^{\alpha,x}_D-\tau^{\alpha,x}_D\wedge\bar\tau^{\alpha,x+\epsilon\xi}_D(\epsilon))\le\frac{2K_0^2}{T}\rightarrow0, \mbox{ as }T\uparrow \infty.$$

To prove (\ref{stopping2}), we just need to notice that for any stopping times $\tau,\gamma_1,\gamma_2$
$$\tau-\tau\wedge\gamma_1\wedge\gamma_2=(\tau-\tau\wedge\gamma_1)I_{\gamma_1<\gamma_2}+(\tau-\tau\wedge\gamma_2)I_{\gamma_1\ge\gamma_2}.$$

By noticing that
$$\psi-\delta=0 \mbox{ on } \partial D_\delta,\qquad \psi-\delta>0,\ \sup_{\alpha\in\mathfrak A}L^\alpha(\psi-\delta)=\sup_{\alpha\in\mathfrak A}L^\alpha\psi\le-1\mbox{ in }D_\delta,$$
we see that the statement is true in the subdomain $D_\delta$.

Similarly, notice that
$$(\psi-\delta_1)(\delta_2-\psi)=0 \mbox{ on } \partial D_{\delta_1}^{\delta_2},\qquad (\psi-\delta_1)(\delta_2-\psi)>0\mbox{ in }D_{\delta_1}^{\delta_2},$$
\begin{align*}
L^\alpha((\psi-\delta_1)(\delta_2-\psi))=&(\delta_1+\delta_2-2\psi)L^\alpha\psi-2(a^\alpha\psi_x,\psi_x)\\
\le&(\delta_1+\delta_2)|L^\alpha\psi|-2|\psi_x^*\sigma^\alpha|^2\mbox{ in }D_{\delta_1}^{\delta_2},\forall\alpha\in\mathfrak A.
\end{align*}
On $\partial D$ it holds that $\psi_x=|\psi_x|n$, where $n(x)$ is the unit inward normal vector at $x\in\partial D$. So due to Assumption \ref{nondeg} and the compactness of $\partial D$,
$$|\psi_x^*\sigma^\alpha|^2=2|\psi_x|^2(a^\alpha n,n)\ge 2|\psi_x|^2\delta_0\ge2\delta_0' \mbox{ on }\partial D,$$
where $\delta_0'$ is a positive constant. By continuity
$$|\psi_x^*\sigma^\alpha|^2\ge\delta_0'\mbox{ in }D_{\delta_1}^{\delta_2},$$
if $\delta_1$ and $\delta_2$ are sufficiently small. It turns out that
$$\sup_{\alpha\in\mathfrak A}L^\alpha\frac{(\psi-\delta_1)(\delta_2-\psi)}{\delta_0'}\le-1,$$
when $\delta_1$ and $\delta_2$ are sufficiently small.
So the statement is still true in the subdomain $D_{\delta_1}^{\delta_2}$ when $\delta_1,\delta_2$ are sufficiently small.

\end{proof}

\section{Proof of Theorem \ref{4c}}

Before proving the main theorem, we state two remarks and one lemma. Remarks \ref{lemma2} and \ref{lemma3} are about two reductions of the problem, and Lemma \ref{lemma4} will be used when estimating the second derivatives. They are nonlinear counterparts of Remarks 3.3 and 3.4 and Lemma 3.2 in \cite{quasi-linear}, and there is no essential change when extending them from linear case to nonlinear case.

\begin{remark}\label{lemma2}
Without loss of generality, we may assume that $c^\alpha\ge1, \forall \alpha\in\mathfrak A$, and replace inequality (\ref{ineq}) by
\begin{equation}
\begin{gathered}
\big\|\sigma^\alpha_{(y)}(x)+(\rho^\alpha(x),y)\sigma^\alpha(x)+\sigma^\alpha(x)Q^\alpha(x,y)\big\|^2+\\
\ 2\big(y,b^\alpha_{(y)}(x)+2(\rho^\alpha(x),y)b^\alpha(x)\big)\le c^\alpha(x)-1+M^\alpha(x)\big(a^\alpha(x)y,y\big).
\end{gathered}
\end{equation}
\end{remark}

\begin{remark}\label{lemma3}
Without loss of generality, we may assume that $v\in C^1({\bar D})$ and $f^\alpha,g\in C^1(\bar D)$ when investigating first derivatives of $v$, and $v\in C^2({\bar D})$ and $f^\alpha,g\in C^2(\bar D)$ when investigating second derivatives of $v$. 
\end{remark}

\begin{lemma}\label{lemma4}
If $f^\alpha, g\in C^2(\bar D)$, and $v\in C^1(\bar D)$, then for any $y\in\partial D$ we have
\begin{equation}
|v_{(n)}(y)|\le K(|g|_{2,D}+\sup_{\alpha\in A}|f^\alpha|_{0,D})\label{normal},
\end{equation}
where $n$ is the unit inward normal on $\partial D$ and the constant $K$ depends only on $K_0$.
\end{lemma}

Let $\delta$ and $\lambda$ be constants satisfying $0<\delta<\lambda^2<\lambda<1$ and that the three sets defined below are nonempty:
\begin{align*}
D_\delta&:=\{x\in D:\delta<\psi(x)\}\\
D_\delta^\lambda&:=\{x\in D:\delta<\psi(x)<\lambda\}\\
D_{\lambda^2}&:=\{x\in D:\lambda^2<\psi(x)\}
\end{align*}
For each $\alpha\in\mathfrak A$, we use the same quasiderivatives $\xi_t^{\alpha,\xi}, \eta_t^{\alpha,\eta}$ and barrier functions $\mathrm B_1(x,\xi), \mathrm B_2(x,\xi)$ constructed in \cite{quasi-linear}. See Remark 3.5 in \cite{quasi-linear} for the motivation of $\mathrm B_1(x,\xi)$ and $\mathrm B_2(x,\xi)$.

Their properties are collected in the following two lemmas.

\begin{lemma} \label{3l1}
In $D_\delta^\lambda$, introduce
$$
\varphi(x)=\lambda^2+\psi(1-\frac{1}{4\lambda}\psi), \ \ \mathrm{B}_1(x,\xi)=\big[\lambda+\sqrt{\psi}(1+\sqrt{\psi})\big]|\xi|^2+K_1\varphi^\frac{3}{2}\pxsop,
$$
where $K_1\in[1,\infty)$ is a constant only depending on $K_0$.

For each $\alpha$, we define the first and second quasiderivatives by (\ref{itoxi}) and (\ref{itoeta}), in which 
\allowdisplaybreaks\begin{align*}
&r(x,\xi):=\rho(x,\xi)+\frac{\psi_{(\xi)}}{\psi},\ \ r_t:=r(x_t,\xi_t),\\
&\mbox{ with }\rho(x,\xi):=-\frac{1}{\Upsilon}\sum_{k=1}^{d_1}\psk(\psk)_{(\xi)},\ \ \Upsilon:=\sum_{k=1}^{d_1}\psk^2;\\
&\hat{r}(x,\xi):=\pxsops,\ \ \hat{r}_t:=\hat{r}(x_t,\xi_t);\\
&\pi^k(x,\xi):=\frac{2\psk\psi_{(\xi)}}{\varphi\psi},\ \  k=1,...,d_1,\ \ \pi_t:=\pi(x_t,\xi_t);\\
&P^{ik}(x,\xi):=\frac{1}{\Upsilon}\big[\psk(\mypsi)_{(\xi)}-\mypsi(\psk)_{(\xi)}\big],\ \  i,k=1,...,d_1,\ \ P_t:=P(x_t,\xi_t); \\
&\hat{\pi}_t^k=\hat{P}_t^{ik}=0,\ \ \forall i,k=1,...d_1,\forall t\in[0,\infty).
\end{align*}
where we drop the superscript $\alpha$ or $\alpha_t$ without confusion. Then (\ref{res1a}), (\ref{res1b1}), (\ref{res1b2}), (\ref{res2a}), (\ref{res2b1}), (\ref{res2b1'}), (\ref{res2b2}) and (\ref{res2b3}) all hold for any constants $p\in (0,\infty)$, $p'\in[0,p)$, $T\in [1,\infty)$, $x\in D_\delta^\lambda$, $\xi,\eta\in\Rd$ and stopping times 
$$\gamma^\alpha\le\tau^{\alpha,x}_{D_\delta^\lambda},\ \gamma^\alpha(\epsilon)\le\tau^{\alpha,x}_{D_\delta^\lambda}\wedge\bar\tau^{\alpha,x+\epsilon\xi}_{D_\delta^\lambda}(\epsilon),\ \gamma_2^\alpha(\epsilon)\le\tau_{D_\delta^\lambda}^{\alpha,x}\wedge\hat\tau_{D_\delta^\lambda}^{\alpha,x(\epsilon)}(\epsilon),$$ 
$$\gamma_3^\alpha(\epsilon)\le\tau_{D_\delta^\lambda}^{\alpha,x}\wedge\hat\tau_{D_\delta^\lambda}^{\alpha,x(\epsilon)}(\epsilon)\wedge\hat\tau_{D_\delta^\lambda}^{\alpha,x(-\epsilon)}(-\epsilon),$$ where $x(\epsilon)=x+\epsilon\xi+\frac{\epsilon^2}{2}\eta$.
 
When $\lambda$ is sufficiently small, for $x\in D^\lambda_\delta$, $\xi\in\Rd$ and $\eta=0$, we have
\begin{enumerate}
\item For each $\alpha\in\mathfrak A$, $\mathrm{B}_1(x^{\alpha,x}_t,\xi^{\alpha,\xi}_t)$ and $\sqrt{\mathrm{B}_1(x^{\alpha,x}_t,\xi^{\alpha,\xi}_t)}$ are local supermartingales on $[0,\tau_1^\delta]$, where $\tau_1^\delta=\tau^{\alpha,x}_{D^\lambda_\delta}$;
\item $\displaystyle{\sup_{\alpha\in\mathfrak A}E^\alpha_{x,\xi}\int_0^{\tau_1^\delta}|\xi_t|^2+\pxtsops dt\le N\mathrm{B}_1(x,\xi)}$;
\item $\displaystyle{\sup_{\alpha\in\mathfrak A}E^\alpha_\xi\sup_{t\le\tau_1^\delta}|\xi_t|^2\le N\mathrm{B}_1(x,\xi)}$;
\item $\displaystyle{\sup_{\alpha\in\mathfrak A}E^\alpha_0|\eta_{\tau_1^\delta}|\le \sup_{\alpha\in\mathfrak A}E^\alpha_0\sup_{t\le\tau_1^\delta}|\eta_t|\le N\mathrm{B}_1(x,\xi)}$;
\item $\displaystyle{\sup_{\alpha\in\mathfrak A}E^\alpha_0\Big(\int_0^{\tau_1^\delta}|\eta_t|^2 dt\Big)^{\frac{1}{2}}\le N\mathrm{B}_1(x,\xi)}$;
\end{enumerate}
where $N$ is a constant depending on $K_0$ and $\epsilon$.
\end{lemma}
\begin{proof}
Notice that $\sup_{\alpha\in A}|\Upsilon^\alpha|_{0,D_\delta^\lambda}$ is bounded from below by a positive constant due to Assumption \ref{nondeg}, so conditions (\ref{cond1}) and (\ref{cond2}) hold with $K_t^\alpha=0$.

The properties (1)-(5) are nothing but Lemma 3.3 in \cite{quasi-linear} because the constant $N$ there doesn't depend on $\alpha.$
\end{proof}

\begin{lemma} \label{3l2}
In $D_{\lambda^2}$, introduce
$$
\mathrm{B}_2(x,\xi)=\lambda^\frac{3}{4}|\xi|^2.
$$

For each $\alpha\in\mathfrak A$, we define the first and second quasiderivatives by (\ref{itoxi}) and (\ref{itoeta}), in which
\allowdisplaybreaks\begin{align*}
&r(x,y):=(\rho(x),y),\ \ r_t:=r(x_t,\xi_t),\ \ \hat{r}_t:=r(x_t,\eta_t),\\
&\pi(x,y):=\frac{M(x)}{2}\sigma^*(x)y,\ \ \pi_t:=\pi(x_t,\xi_t),\ \ \hat{\pi}_t:=\pi(x_t,\eta_t),\\
&P(x,y):=Q(x,y),\ \ P_t:=P(x_t,\xi_t),\ \ \hat{P}_t:=P(x_t,\eta_t).
\end{align*}
where $\rho(x)$, $M(x)$ and $Q(x,y)$ are defined in the statement of the main theorem and satisfy the inequality (\ref{ineq}), and again, we drop the superscript $\alpha$ or $\alpha_t$ without confusion. Then (\ref{res1a}), (\ref{res1b1}), (\ref{res1b2}), (\ref{res2a}), (\ref{res2b1}), (\ref{res2b1'}), (\ref{res2b2}) and (\ref{res2b3}) all hold for any constants $p\in (0,\infty)$, $p'\in[0,p)$, $T\in [1,\infty)$, $x\in D_\delta^\lambda$, $\xi,\zeta\in\Rd$ and stopping times 
$$\gamma^\alpha\le\tau^{\alpha,x}_{D_{\lambda^2}}, \ \gamma^\alpha(\epsilon)\le\tau^{\alpha,x}_{D_{\lambda^2}}\wedge\bar\tau^{\alpha,x+\lambda\xi}_{D_{\lambda^2}}, \ \gamma_2^\alpha(\epsilon)\le\tau_{D_{\lambda^2}}^{\alpha,x}\wedge\hat\tau_{D_{\lambda^2}}^{\alpha,x+\epsilon\xi+\frac{\epsilon^2}{2}\eta}(\epsilon),$$ 
$$\gamma_3^\alpha(\epsilon)\le\tau_{D_{\lambda^2}}^{\alpha,x}\wedge\hat\tau_{D_{\lambda^2}}^{\alpha,x+\epsilon\xi+\frac{\epsilon^2}{2}\eta}(\epsilon)\wedge\hat\tau_{D_{\lambda^2}}^{\alpha,x-\epsilon\xi+\frac{\epsilon^2}{2}\eta}(-\epsilon),$$ where $x(\epsilon)=x+\epsilon\xi+\frac{\epsilon^2}{2}\eta$.

Furthermore, for $x\in D_{\lambda^2}$, $\xi\in\Rd$ and $\eta=0$, we have
\begin{enumerate}
\item $e^{-\phi^{\alpha,x}_t}\mathrm{B}_2(x^{\alpha,x}_t,\xi^{\alpha,\xi}_t)$ and $\sqrt{e^{-\phi^{\alpha,x}_t}\mathrm{B}_2(x^{\alpha,x}_t,\xi^{\alpha,\xi}_t)}$ are local supermartingales on $[0,\tau_2)$, where $\tau_2=\tau_{D_{\epsilon^2}}^{\alpha,x}$.
\item $\displaystyle{\sup_{\alpha\in\mathfrak A}E^\alpha_{x,\xi}\int_0^{\tau_2} e^{-\phi_t}|\xi_t|^2 dt\le N\mathrm{B}_2(x,\xi)}$
\item $\displaystyle{\sup_{\alpha\in\mathfrak A}E^\alpha_{x,\xi}\sup_{t\le\tau_2}e^{-\phi_t}|\xi_t|^2\le N\mathrm{B}_2(x,\xi)}$
\item $\displaystyle{\sup_{\alpha\in\mathfrak A}E^\alpha_{x,0}e^{-\phi_{\tau_2}}|\eta_{\tau_2}|\le \sup_{\alpha\in\mathfrak A}E^\alpha_{x,0}\sup_{t\le\tau_2}e^{-\phi_t}|\eta_t|\le N\mathrm{B}_2(x,\xi)}$
\item $\displaystyle{\sup_{\alpha\in\mathfrak A}E^\alpha_{x,0}\Big(\int_0^{\tau_2} e^{-2\phi_t}|\eta_t|^2 dt\Big)^\frac{1}{2}\le N\mathrm{B}_2(x,\xi)}$
\item The above inequalities are still all true if we replace $\phi^{\alpha,x}_t$ by $\phi^{\alpha,x}_t-\frac{1}{2}t$. More precisely, we have
$$\sup_{\alpha\in\mathfrak A}E^\alpha_{x,\xi}\int_0^{\tau_2} e^{-\phi_t+\frac{1}{2}t}|\xi_t|^2 dt\le N\mathrm{B}_2(x,\xi),\ \sup_{\alpha\in\mathfrak A}E^\alpha_{\xi,0}\sup_{t\le\tau_2}e^{-\phi_t+\frac{1}{2}t}|\xi_t|^2\le N\mathrm{B}_2(x,\xi)$$
$$\sup_{\alpha\in\mathfrak A}E^\alpha_{x,0}\Big(\int_0^{\tau_2} e^{-2\phi_t+t}|\eta_t|^2 dt\Big)^\frac{1}{2}\le N\mathrm{B}_2(x,\xi),\ \sup_{\alpha\in\mathfrak A}E^\alpha_{x,0}\sup_{t\le\tau_2}e^{-\phi_t+\frac{1}{2}t}|\eta_t|\le N\mathrm{B}_2(x,\xi)$$
\end{enumerate}
where $N$ is constant depending on $K_0$ and $\lambda$.
\end{lemma}
\begin{proof}
The same as the proof of Lemma \ref{3l1}.
\end{proof}

We split the proof of Theorem \ref{4c} into three parts. Note that in the proof, for simplicity of notation, we may drop the superscripts such as $\alpha$ when it will cause no confusion.

\begin{proof}[Proof of (\ref{4d})] 
First, we fix an $x\in D_\delta^\lambda$ and a $\xi\in\Rd\setminus\{{0}\}$. Choose $ \epsilon_0>0$ sufficiently small, so that $B(x,\epsilon_0|\xi|):=\{y:|y-x|\le \epsilon_0|\xi|\}\subset D_\delta^\lambda$. For any $ \epsilon\in(0, \epsilon_0)$, by Bellman's principle (Theorem 1.1 in \cite{MR2303952}, in which $Q$ is defined by $D\times [-1,T+1]$, where $T$ is an arbitrary positive constant), we have, 
\allowdisplaybreaks\begin{align*}
\frac{v(x+ \epsilon\xi)-v(x)}{ \epsilon}=&\frac{1}{ \epsilon}\bigg\{\sup_{\alpha\in \mathfrak{A}}E^\alpha_{x+ \epsilon\xi}\Big[v(x_{\gamma})e^{-\phi_{\gamma}}+\int_0^{\gamma}f^{\alpha_s}(x_s)e^{-\phi_s}ds\Big]\\
&-\sup_{\alpha\in \mathfrak{A}}E^\alpha_x\Big[v(x_{\gamma})e^{-\phi_{\gamma}}+\int_0^{\gamma}f^{\alpha_s}(x_s)e^{-\phi_s}ds\Big]\bigg\},
\end{align*}
where the stopping time $\gamma^\alpha\le\tau_{D_\delta^\lambda}^{\alpha,x+ \epsilon\xi}\wedge\tau_{D_\delta^\lambda}^{\alpha,x}\wedge T$.

By Theorem 2.1 in \cite{MR631436} and Lemmas 2.1 and 2.2 in \cite{MR637615}, 
\begin{equation*}
\begin{gathered}
\sup_{\alpha\in \mathfrak{A}}E^\alpha_{x+ \epsilon\xi}\Big[v(x_{\gamma})e^{-\phi_{\gamma }}+\int_0^{\gamma }f^{\alpha_s}(x_s)e^{-\phi_s}ds\Big]\\
=\sup_{\alpha\in\mathfrak{A}}E^\alpha_{x+\epsilon\xi}\Big[v(y_{\gamma}( \epsilon))p_{\gamma}( \epsilon) e^{-{\phi}_{\gamma}( \epsilon)}+\int_0^{\gamma}(1+2\epsilon r_s)f^{\alpha_s}(y_s( \epsilon))p_s( \epsilon)e^{-\phi_s( \epsilon)}ds\Big],
\end{gathered}
\end{equation*}
in which $y_t^{\alpha,y}( \epsilon)$ is the solution to the It\^o equation (\ref{itoy}),
$$
\phi_t^{\alpha,y}( \epsilon):=\int_0^t(1+2 \epsilon r_s^{\alpha})c^{\alpha_s}(y_s^{\alpha,y}( \epsilon))ds,
$$
\begin{equation}\label{pt}
p_t^\alpha( \epsilon):=\exp\bigg(\int_0^t \epsilon\pi^{\alpha}_sdw_s-\frac{1}{2}\int_0^t| \epsilon\pi^{\alpha}_s|^2ds\bigg).
\end{equation}
with $\alpha\in\mathfrak A$, $r_s^{\alpha}, \pi_s^{\alpha}, P_s^{\alpha}$ defined in Lemma \ref{3l1}, and $\gamma^\alpha\le\bar\tau_{D_\delta^\lambda}^{\alpha,x+ \epsilon\xi}\wedge\tau_{D_\delta^\lambda}^{\alpha,x}\wedge T$.

Let
\begin{align*}
&q_t^\alpha(\epsilon)=\int_0^{t}(1+2\epsilon r^\alpha_s)f^{\alpha_s}(y_s( \epsilon))p_s( \epsilon)e^{-\phi_s( \epsilon)}ds,\\
&\bar y_t^{\alpha,y}(\epsilon)=(y_t^{\alpha,y}(\epsilon), -\phi_t^\alpha(\epsilon),p_t^\alpha(\epsilon),q_t^\alpha(\epsilon)),\\
&\bar x_t^{\alpha,x}=(x_t^{\alpha,x},-\phi_t^\alpha(0),p_t^\alpha(0),q_t^\alpha(0)).
\end{align*}

For any $\bar x=(x,x^{d+1},x^{d+2},x^{d+3})\in D\times\mathbb{R}^-\times\mathbb R^+\times\mathbb R$, introduce
\begin{equation}\label{V}
V(\bar x)=v(x)\exp(x^{d+1})x^{d+2}+x^{d+3}.
\end{equation}
Then we have
$$\frac{v(x+ \epsilon\xi)-v(x)}{ \epsilon}=\frac{1}{ \epsilon}\Big(\sup_{\alpha\in\mathfrak{A}}E^\alpha_{x+\epsilon\xi}V(\bar y_\gamma(\epsilon))-\sup_{\alpha\in\mathfrak{A}}E^\alpha_xV(\bar x_\gamma)\Big),$$
in which we let
$$\gamma=\gamma^\alpha(\epsilon,n,T)=\bar\tau_{D_\delta^\lambda}^{\alpha,x+ \epsilon\xi}\wedge\tau_{D_\delta^\lambda}^{\alpha,x}\wedge\kappa_n^\alpha\wedge T,$$
where
$$\kappa_n^\alpha=\inf\{t\ge0:|\xi_t^{\alpha,\xi}|\ge n\}.$$

Since the difference of two supremums is less than the supremum of the differences, and the supremum of a sum is less than the sum of the supremums, we have
\begin{align*}
\frac{v(x+ \epsilon\xi)-v(x)}{ \epsilon}\le&\sup_{\alpha\in\mathfrak{A}}E\frac{V(\bar y^{\alpha,x+\epsilon\xi}_{\gamma^\alpha}(\epsilon))-V(\bar x^{\alpha,x}_{\gamma^\alpha})}{ \epsilon}\\
\le&\sup_{\alpha\in\mathfrak{A}}E\frac{V(\bar y^{\alpha,x+\epsilon\xi}_{\gamma^\alpha}(\epsilon))-V(\bar x^{\alpha,x}_{\gamma^\alpha})}{ \epsilon}-V_{(\bar\xi_{\gamma^\alpha}^{\alpha,\xi})}(\bar x_{\gamma^\alpha}^{\alpha,x})+\sup_{\alpha\in\mathfrak{A}}EV_{(\bar\xi_{\gamma^\alpha}^{\alpha,\xi})}(\bar x_{\gamma^\alpha}^{\alpha,x})\\
:=&I_1(\epsilon,n,T)+I_2(\epsilon,n,T),
\end{align*}
where 
\begin{equation}\label{xibar}
\bar\xi_t^{\alpha,\xi}=(\xi_t^{\alpha,\xi},\xi_t^{d+1,\alpha},\xi_t^{d+2,\alpha},\xi_t^{d+3,\alpha}),
\end{equation}
with
\begin{align*}
&\xi_t^{d+1,\alpha}:=-\int_0^t\Big[c^{\alpha_s}_{(\xi^{\alpha,\xi}_s)}(x^{\alpha,x}_s)+2r^\alpha_sc^{\alpha_s}(x^{\alpha,x}_s)\Big] ds,\\
&\xi_t^{d+2,\alpha}:=\xi^{0,\alpha}_t=\int_0^t\pi^\alpha_sdw_s,\\
&\xi_t^{d+3,\alpha}:=\int_0^te^{-\phi_s^{\alpha,x}}\Big[f^{\alpha_s}_{(\xi_s^{\alpha,\xi})}(x_s^{\alpha,x})+\big(2r^\alpha_s+\xi_s^{d+1,\alpha}+\xi_s^{d+2,\alpha}\big)f^{\alpha_s}(x_s^{\alpha,x})\Big]ds.
\end{align*}

We claim that 
\begin{equation}\label{I1}
\lim_{\epsilon\downarrow0}I_1(\epsilon,n,T)=0.
\end{equation}
To show it, bearing in mind that for any $h^\alpha(x)\in C^1(\bar D_\delta)$, whose derivatives are uniformly continuous in $\alpha$, 
we have, for any $x,y\in D_\delta$ and $\xi\in \Rd$, $r\in\Rd$ and $n\in\mathbb N$,
\begin{align}
&|\frac{(1+2\epsilon r)h^\alpha(y)-h^\alpha(x)}{\epsilon}-h^\alpha_{(\xi)}(x)-2rh^\alpha(x)|\label{h}\\
\nonumber=&|h^\alpha_{(\frac{y-x}{\epsilon})}(y^*)-h^\alpha_{(\xi)}(x)+2r(h^\alpha(y)-h^\alpha(x))|\\
\nonumber\le&|(h^\alpha_x(y^*)-h^\alpha_x(x),\frac{y-x}{\epsilon})|+|(h^\alpha_x(x),\frac{y-x}{\epsilon}-\xi)|+2K_0(\epsilon r^2+\frac{|y-x|^2}{\epsilon}),
\end{align}
where $y^*$ is a point on the line segment with ending points $x$ and $y$.

First, by Theorem \ref{thm1}, for any contants $p$ and $p'$ satisfying $0\le p'<p<\infty$, we have
\begin{equation}
\sup_{\alpha\in\mathfrak A}E^\alpha_\xi\sup_{t\le\gamma}|\xi_t|^p<\infty,
\end{equation}
\begin{equation}
\lim_{\epsilon\downarrow0}\sup_{\alpha\in\mathfrak A}E^\alpha\sup_{t\le\gamma}\frac{|y_t^{x+\epsilon\xi}(\epsilon)-x_t^{x}|^p}{\epsilon^{p'}}=0,
\end{equation}
\begin{equation}\label{d}
\lim_{\epsilon\downarrow0}\sup_{\alpha\in\mathfrak A}E^\alpha\sup_{t\le\gamma}|\frac{y_t^{x+\epsilon\xi}(\epsilon)-x_t^{x}}{\epsilon}-\xi_t^{\xi}|^p=0.
\end{equation}

Second, apply (\ref{h}) to $c^\alpha(x)$ we get
\begin{equation}\label{d+1}
\varlimsup_{\epsilon\downarrow0}\sup_{\alpha\in\mathfrak A}E^\alpha\sup_{t\le\gamma}|\frac{\phi_t(0)-\phi_t(\epsilon)}{\epsilon}-\xi_t^{d+1}|^p=0.
\end{equation}

Third, we notice that
$$\frac{p_t(\epsilon)-p_t(0)}{\epsilon}=\frac{p_t(\epsilon)-1}{\epsilon}=\int_0^tp_s(\epsilon)\pi_sdw_s.$$
Recall that $\gamma^\alpha\le \kappa^\alpha_n\wedge T$. It follows that
\begin{align*}
E^\alpha\sup_{t\le\gamma}|\frac{p_t(\epsilon)-1}{\epsilon}-\xi_t^{d+2}|^p\le&N(p)E^\alpha\Big(\int_0^\gamma(p_t(\epsilon)-1)^2|\pi_t|^2dt\Big)^{p/2}\\
\le&\epsilon^pN(p) E^\alpha\Big(\sup_{t\le\gamma}\Big|\frac{p_t(\epsilon)-1}{\epsilon}\Big|^{2p}+\int_0^\gamma|\pi_t|^{2p}dt\Big)\\
\le&\epsilon^p N(p)E^\alpha\Big(\int_0^\gamma p_t^{2p}(\epsilon)|\pi_t|^{2p}dt+\int_0^\gamma|\pi_t|^{2p}dt\Big).
\end{align*}
Hence
\begin{equation}\label{d+2}
\lim_{\epsilon\downarrow0}\sup_{\alpha\in\mathfrak A}E^\alpha\sup_{t\le\gamma}|\frac{p_t(\epsilon)-p_t(0)}{\epsilon}-\xi_t^{d+2}|^p=0.
\end{equation}

Fourth, bearing in mind that
\begin{align*}
&|\frac{f(\epsilon)g(\epsilon)-fg}{\epsilon}-f'g-fg'|\\
\le&|\frac{f(\epsilon)-f}{\epsilon}-f'||g(\epsilon)|+|\frac{g(\epsilon)-g}{\epsilon}-g'||f|+|f'||g(\epsilon)-g|\\
\le&|\frac{f(\epsilon)-f}{\epsilon}-f'||g(\epsilon)|+|\frac{g(\epsilon)-g}{\epsilon}-g'||f|+\epsilon(|f'|^2+\frac{|g(\epsilon)-g|^2}{\epsilon^2}).
\end{align*}
Therefore, to prove
\begin{equation}\label{d+3}
\lim_{\epsilon\downarrow0}\sup_{\alpha\in\mathfrak A}E^\alpha\sup_{t\le\gamma}|\frac{q_t(\epsilon)-q_t(0)}{\epsilon}-\xi_t^{d+3}|^p=0,
\end{equation}
it suffices to show that
$$\lim_{\epsilon\downarrow0}\sup_{\alpha\in\mathfrak A}E^\alpha\sup_{t\le\gamma}|\frac{(1+2\epsilon r_t)f^{\alpha_t}(y_t(\epsilon))-f^{\alpha_t}(x_t)}{\epsilon}-f^{\alpha_t}_{(\xi_t)}(x_t)-2r_tf^{\alpha_t}(x_t)|^p=0,$$
$$\lim_{\epsilon\downarrow0}\sup_{\alpha\in\mathfrak A}E^\alpha\sup_{t\le\gamma}|\frac{e^{-\phi_t(\epsilon)}-e^{-\phi_t(0)}}{\epsilon}+\xi_t^{d+1}e^{-\phi_t(0)}|^p=0.$$
The first equation is true due to (\ref{h}) with $h^\alpha=f^\alpha$. The second one is true by a similar argument.

Finally, observe that for any $\bar x=(x,x^{d+1},1,x^{d+3})$, $\bar y=(y,y^{d+1},y^{d+2},y^{d+3})$, $\bar \xi=(\xi,\xi^{d+1},\xi^{d+2},\xi^{d+3})\in D\times\mathbb R^-\times \mathbb R^+\times\mathbb R$, we have
\begin{align*}
\frac{V(\bar y)-V(\bar x)}{\epsilon}-V_{(\bar\xi)}(\bar x)=&\frac{v(y)e^{y^{d+1}}y^{d+2}-v(x)e^{x^{d+1}}}{\epsilon}+\frac{y^{d+3}-x^{d+3}}
{\epsilon}\\
&-e^{x^{d+1}}[v_{(\xi)}(x)+v(x)(\xi^{d+1}+\xi^{d+2})]-\xi^{d+3}.
\end{align*}
It is not hard to see (\ref{I1}) is true with (\ref{d}), (\ref{d+1}), (\ref{d+2}) and (\ref{d+3}) in hand.

To estimate $I_2(\epsilon, n,T)$, we notice that $V_{(\bar\xi^{\alpha,\xi}_t)}(\bar x^{\alpha,x}_t)$ is exactly $X^\alpha_t$ defined by (2.9) in \cite{quasi-linear}, in which $u$ is replaced by $v$. More precisely,
\begin{align*}
V_{(\bar\xi^{\alpha,\xi}_t)}(\bar x^{\alpha,x}_t)=X_t^{\alpha}:=&e^{-\phi_t^{\alpha,x}}\Big[v_{(\xi_t^{\alpha,\xi})}(x_t^{\alpha,x})+\tilde{\xi}_t^{0,\alpha}v(x_t^{\alpha,x})\Big]\\
&+\int_0^te^{-\phi_s^{\alpha,x}}\Big[f^{\alpha_s}_{(\xi_s^{\alpha,\xi})}(x_s^{\alpha,x})+\big(2r^\alpha_s+\tilde{\xi}_s^{0,\alpha}\big)f^{\alpha_s}(x_s^{\alpha,x})\Big]ds,
\end{align*}
where
$$\tilde\xi^{0,\alpha}_t=\xi^{0,\alpha}_t+\xi_t^{d+1,\alpha}.$$

It follows that
$$I_2(\epsilon,n, T)=\sup_{\alpha\in\mathfrak A}E^\alpha X_\gamma\le \sup_{\alpha\in\mathfrak A}E e^{-\phi_{\gamma^\alpha}^{\alpha,x}}v_{(\xi_{\gamma^\alpha}^{\alpha,\xi})}(x_{\gamma^\alpha}^{\alpha,x})+\sup_{\alpha\in\mathfrak A}E\Big(X^\alpha_{\gamma^\alpha}-e^{-\phi_{\gamma^\alpha}^{\alpha,x}}v_{(\xi_{\gamma^\alpha}^{\alpha,\xi})}(x_{\gamma^\alpha}^{\alpha,x})\Big).$$

We first notice that as in the proof of (3.4) in \cite{quasi-linear}, for each $\alpha$,
\begin{align*}
&E\sup_{t\le\tau^{\alpha,x}_{D_\delta^\lambda}}\Big(X^\alpha_t-e^{-\phi_{t}^{\alpha,x}}v_{(\xi_{t}^{\alpha,\xi})}(x_{t}^{\alpha,x})\Big)\\
=&E\sup_{t\le\tau^{\alpha,x}_{D_\delta^\lambda}}\bigg\{e^{-\phi_t^{\alpha,x}}\tilde{\xi}_t^{0,\alpha}v(x_t^{\alpha,x})+\int_0^te^{-\phi_s^{\alpha,x}}\Big[f^{\alpha_s}_{(\xi_s^{\alpha,\xi})}(x_s^{\alpha,x})+\big(2r^\alpha_s+\tilde{\xi}_s^{0,\alpha}\big)f^{\alpha_s}(x_s^{\alpha,x})\Big]ds\bigg\}\\
\le&(|g|_{0,D}+|f^\alpha|_{0,D})\Big(E\sup_{t\le\tau^{\alpha,x}_{D_\delta^\lambda}}|\xi_t^{0,\alpha}|+E\sup_{t\le\tau^{\alpha,x}_{D_\delta^\lambda}}|\xi_t^{d+1,\alpha}|\Big)\\
&+|f^\alpha|_{1,D}\Big(E\int_0^{\tau^{\alpha,x}_{D_\delta^\lambda}}|\xi_s^{\alpha,\xi}|+2r_s^\alpha ds+E\sup_{t\le\tau^{\alpha,x}_{D_\delta^\lambda}}|\xi_t^{0,\alpha}|+E\sup_{t\le\tau^{\alpha,x}_{D_\delta^\lambda}}|\xi_t^{d+1,\alpha}|\Big)
\end{align*}
Repeat the estimates (3.19)-(3.21) in \cite{quasi-linear}, we have
$$E\sup_{t\le\tau^{\alpha,x}_{D_\delta^\lambda}}\Big(X^\alpha_t-e^{-\phi_{t}^{\alpha,x}}v_{(\xi_{t}^{\alpha,\xi})}(x_{t}^{\alpha,x})\Big)\le N\sqrt{\mathrm B_1(x,\xi)}$$
where $N$ is independent of $\alpha$. So
$$I_2(\epsilon,n, T)\le \sup_{\alpha\in\mathfrak A}E e^{-\phi_{\gamma^\alpha}^{\alpha,x}}v_{(\xi_{\gamma^\alpha}^{\alpha,\xi})}(x_{\gamma^\alpha}^{\alpha,x})+N\sqrt{\mathrm B_1(x,\xi)}.$$

We next notice that
\begin{align*}
\sup_{\alpha\in\mathfrak{A}}Ev_{(\xi_{\gamma^\alpha}^{\alpha,\xi})}(x_{\gamma^\alpha}^{\alpha,x})=&\sup_{\alpha\in\mathfrak{A}}E\frac{v_{(\xi_{\gamma^\alpha}^{\alpha,\xi})}(x_{\gamma^\alpha}^{\alpha,x})}{\sqrt{\mathrm{B}_1(x_{\gamma^\alpha}^{\alpha,x},\xi_{\gamma^\alpha}^{\alpha,\xi})}}\cdot\sqrt{\mathrm{B}_1(x_{\gamma^\alpha}^{\alpha,x},\xi_{\gamma^\alpha}^{\alpha,\xi})}\\
\le&\sup_{\alpha\in\mathfrak{A}}E\bigg(\frac{v_{(\xi_{\gamma^\alpha}^{\alpha,\xi})}(x_{\gamma^\alpha}^{\alpha,x})}{\sqrt{\mathrm{B}_1(x_{\gamma^\alpha}^{\alpha,x},\xi_{\gamma^\alpha}^{\alpha,\xi})}}-\frac{v_{(\xi_{\gamma^\alpha}^{\alpha,\xi})}(x_{\tau^{\alpha,x}_{D_\delta^\lambda}}^{\alpha,x})}{\sqrt{\mathrm{B}_1(x_{\tau^{\alpha,x}_{D_\delta^\lambda}}^{\alpha,x},\xi_{\gamma^\alpha}^{\alpha,\xi})}}\bigg)\cdot\sqrt{\mathrm{B}_1(x_{\gamma^\alpha}^{\alpha,x},\xi_{\gamma^\alpha}^{\alpha,\xi})}\\
&+\sup_{\alpha\in\mathfrak{A}}E\frac{v_{(\xi_{\gamma^\alpha}^{\alpha,\xi})}(x_{\tau^{\alpha,x}_{D_\delta^\lambda}}^{\alpha,x})}{\sqrt{\mathrm{B}_1(x_{\tau^{\alpha,x}_{D_\delta^\lambda}}^{\alpha,x},\xi_{\gamma^\alpha}^{\alpha,\xi})}}\cdot\sqrt{\mathrm{B}_1(x_{\gamma^\alpha}^{\alpha,x},\xi_{\gamma^\alpha}^{\alpha,\xi})}\\
:=&J_1(\epsilon,n,T)+J_2(\epsilon,n,T).
\end{align*}
Notice that
$$\frac{v_{(\xi)}(x)}{\sqrt{\mathrm{B}_1(x,\xi)}}=\frac{v_{(\xi/|\xi|)}(x)}{\sqrt{\mathrm{B}_1(x,\xi/|\xi|)}}$$
is a continuous function from $D_\delta^\lambda\times S_1$ to $\mathbb{R}$, where $S_1$ is the unit sphere in $\Rd$. By Weierstrass Approximation Theorem,  there exists a polynomial $W(x,\xi): D_\delta^\lambda\times S_1\rightarrow\mathbb{R}$, such that 
$$\sup_{x\in D_\delta^\lambda,\xi\in S_1}|\frac{v_{(\xi)}(x)}{\sqrt{\mathrm{B}_1(x,\xi)}}-W(x,\xi)|\le1.$$
It follows that
\begin{align*}
J_1(\epsilon,n,T)\le&\sup_{\alpha\in\mathfrak{A}}E|W(x_{\gamma^\alpha}^{\alpha,x},\xi_{\gamma^\alpha}^{\alpha,\xi})-W(x_{\tau^{\alpha,x}_{D_\delta^\lambda}}^{\alpha,x},\xi_{\gamma^\alpha}^{\alpha,\xi})|\sqrt{\mathrm{B}_1(x_{\gamma^\alpha}^{\alpha,x},\xi_{\gamma^\alpha}^{\alpha,\xi})}\\
&+2\sup_{\alpha\in\mathfrak{A}}E\sqrt{\mathrm{B}_1(x_{\gamma^\alpha}^{\alpha,x},\xi_{\gamma^\alpha}^{\alpha,\xi})}\\
\le&N\sup_{\alpha\in\mathfrak{A}}E|x_{\gamma^\alpha}^{\alpha,x}-x_{\tau^{\alpha,x}_{D_\delta^\lambda}}^{\alpha,x}|\sqrt{\mathrm{B}_1(x_{\gamma^\alpha}^{\alpha,x},\xi_{\gamma^\alpha}^{\alpha,\xi})}+2\sqrt{\mathrm{B}_1(x,\xi)}\\
\le&N\sqrt{\mathrm{B}_1(x,\xi)}\sup_{\alpha\in\mathfrak{A}}E|x_{\gamma^\alpha}^{\alpha,x}-x_{\tau^{\alpha,x}_{D_\delta^\lambda}}^{\alpha,x}|^2+\frac{\sup_{\alpha\in\mathfrak{A}}E\mathrm{B}_1(x_{\gamma^\alpha}^{\alpha,x},\xi_{\gamma^\alpha}^{\alpha,\xi})}{\sqrt{\mathrm{B}_1(x,\xi)}}\\
&+2\sqrt{\mathrm{B}_1(x,\xi)}\\
\le&N\sqrt{\mathrm{B}_1(x,\xi)}E|x_{\gamma^\alpha}^{\alpha,x}-x_{\tau^{\alpha,x}_{D_\delta^\lambda}}^{\alpha,x}|^2+3\sqrt{\mathrm{B}_1(x,\xi)}\\
\le&N\sqrt{\mathrm{B}_1(x,\xi)}\Big(E|\tau^{\alpha,x}_{D_\delta^\lambda}-\tau^{\alpha,x}_{D_\delta^\lambda}\wedge\tau^{\alpha,x+\epsilon\xi}_{D_\delta^\lambda}|+E|{\tau^{\alpha,x}_{D_\delta^\lambda}}-{\tau^{\alpha,x}_{D_\delta^\lambda}\wedge T}|\\
&+|{\tau^{\alpha,x}_{D_\delta^\lambda}}-{\tau^{\alpha,x}_{D_\delta^\lambda}}\wedge\kappa_n^\alpha|\Big)+3\sqrt{\mathrm{B}_1(x,\xi)}.
\end{align*}
Thus
$$\varlimsup_{T\uparrow\infty}\varlimsup_{n\uparrow\infty}\varlimsup_{\epsilon\downarrow0}J_1(\epsilon,n,T)\le3\sqrt{\mathrm{B}_1(x,\xi)}.$$
Also, notice that
\begin{align*}
J_2(\epsilon,n,T)\le&\sup_{y\in \partial D_\delta^\lambda,\zeta\in \Rd\setminus\{0\}}\frac{v_{(\zeta)}(y)}{\sqrt{\mathrm{B}_1(y,\zeta)}}\sup_{\alpha\in\mathfrak{A}}E\sqrt{\mathrm{B}_1(x_{\gamma^\alpha}^{\alpha,x},\xi_{\gamma^\alpha}^{\alpha,\xi})}\\
\le&\sup_{y\in \partial D_\delta^\lambda,\zeta\in  \Rd\setminus\{0\}}\frac{v_{(\zeta)}(y)}{\sqrt{\mathrm{B}_1(y,\zeta)}}\cdot\sqrt{\mathrm{B}_1(x,\xi)}.
\end{align*}
Hence,
$$\varlimsup_{T\uparrow\infty}\varlimsup_{n\uparrow\infty}\varlimsup_{\epsilon\downarrow0}I_2(\epsilon,n,T)\le\sup_{y\in \partial D_\delta^\lambda,\zeta\in  \Rd\setminus\{0\}}\frac{v_{(\zeta)}(y)}{\sqrt{\mathrm{B}_1(y,\zeta)}}\cdot\sqrt{\mathrm{B}_1(x,\xi)}+N\sqrt{\mathrm{B}_1(x,\xi)}.$$
We conclude that
$$\frac{v_{(\xi)}(x)}{\sqrt{\mathrm{B}_1(x,\xi)}}\le\sup_{y\in \partial D_\delta^\lambda,\zeta\in  \Rd\setminus\{0\}}\frac{v_{(\zeta)}(y)}{\sqrt{\mathrm{B}_1(y,\zeta)}}+N,\ \forall x\in D_\delta^\lambda,\xi\in\Rd\setminus\{0\}.$$
Notice that $\mathrm B_1(x,\xi)=\mathrm B_1(x,-\xi)$. Replacing $\xi$ by $-\xi$, we have
$$\frac{-v_{(\xi)}(x)}{\sqrt{\mathrm{B}_1(x,\xi)}}\le\sup_{y\in \partial D_\delta^\lambda,\zeta\in  \Rd\setminus\{0\}}\frac{v_{(\zeta)}(y)}{\sqrt{\mathrm{B}_1(y,\zeta)}}+N,\ \forall x\in D_\delta^\lambda,\xi\in\Rd\setminus\{0\},$$
which implies that 
\begin{equation}\label{ind1}
\frac{|v_{(\xi)}(x)|}{\sqrt{\mathrm{B}_1(x,\xi)}}\le\sup_{y\in \partial D_\delta^\lambda,\zeta\in  \Rd\setminus\{0\}}\frac{|v_{(\zeta)}(y)|}{\sqrt{\mathrm{B}_1(y,\zeta)}}+N,\ \forall x\in D_\delta^\lambda,\xi\in\Rd\setminus\{0\}.
\end{equation}
Repeating the argument above in $D_{\lambda^2}$, we have
\begin{equation}\label{ind2}
\frac{|v_{(\xi)}(x)|}{\sqrt{\mathrm{B}_2(x,\xi)}}\le\sup_{y\in \partial D_{\lambda^2},\zeta\in  \Rd\setminus\{0\}}\frac{|v_{(\zeta)}(y)|}{\sqrt{\mathrm{B}_2(y,\zeta)}}+N,\ \forall x\in D_{\lambda^2},\xi\in\Rd\setminus\{0\}.
\end{equation}
The inequalities (\ref{ind1}) and (\ref{ind2}) are the same as (3.22) and (3.24) in \cite{quasi-linear}. So by repeating the argument after (3.24) in \cite{quasi-linear}, we get
$$v_{(\xi)}(x)\le N\bigg(|\xi|+\frac{|\psi_{(\xi)}(x)|}{\psi^{\frac{1}{2}}(x)}\bigg), \ \ \mbox{ a.e. in }D.$$
(\ref{4d}) is proved.
\end{proof}

\begin{proof}[Proof of (\ref{4dd})]

The idea is the same as the first order case. Fix $x\in D_\delta^\lambda$, $\xi\in\Rd\setminus\{{0}\}$ and sufficiently small positive $ \epsilon_0$, so that $B(x,\epsilon_0|\xi|)\subset D_\delta^\lambda$. For each $\alpha\in\mathfrak A$, let $\gamma^\alpha:=\hat\tau_{D_\delta^\lambda}^\alpha(x+ \epsilon\xi)\wedge\tau_{D_\delta^\lambda}^\alpha(x)\wedge\hat\tau_{D_\delta^\lambda}^\alpha(x- \epsilon\xi)\wedge\kappa_n^\alpha\wedge T$, where $T\in [1,\infty)$. We have
\allowdisplaybreaks\begin{align*}
&-\frac{v(x+ \epsilon\xi)-2v(x)+v(x- \epsilon\xi)}{ \epsilon^2}\\
=&\frac{1}{ \epsilon^2}\bigg\{-\sup_{\alpha\in \mathfrak{A}}E^\alpha_{x+ \epsilon\xi}\Big[v(z_{\gamma}(\epsilon))\hat p_{\gamma}(\epsilon) e^{-\hat\phi_{\gamma}(\epsilon)}+\int_0^{\gamma}(1+2\epsilon r_s+\epsilon^2\hat r_s)f^{\alpha_s}(z_s(\epsilon))\hat p_s(\epsilon)e^{-\hat\phi_s(\epsilon)}ds\Big]\\
&\qquad +2\sup_{\alpha\in \mathfrak{A}}E^\alpha_{x}\Big[v(x_{\gamma })\hat p_{\gamma } e^{-\hat\phi_{{\gamma }}}+\int_0^{{\gamma }}f^{\alpha_s}(x_s)\hat p_se^{-\hat\phi_s}ds\Big]\\
&\qquad -\sup_{\alpha\in \mathfrak{A}}E^\alpha_{x- \epsilon\xi}\Big[v(z_{\gamma}(-\epsilon))\hat p_{\gamma}(-\epsilon) e^{-\hat\phi_{\gamma}(-\epsilon)}\\
&\qquad+\int_0^{{\gamma}}(1-2\epsilon r_s+\epsilon^2\hat r_s)f^{\alpha_s}(z_s(-\epsilon))\hat p_s(-\epsilon)e^{-\hat\phi_s(-\epsilon)}ds\Big]\bigg\},
\end{align*}
in which $z_t^{\alpha,z}(\epsilon)$ is the solution to the It\^o equation (\ref{itoz}),
$$
\hat\phi_t^{\alpha,z}( \epsilon):=\int_0^t(1+2 \epsilon r_s^{\alpha}+\epsilon^2 \hat r_s^\alpha)c^{\alpha_s}(z_s^{\alpha,z}( \epsilon))ds,
$$
and
$$
\hat p_t^\alpha( \epsilon):=\exp\bigg(\int_0^t (\epsilon\pi^{\alpha}_s+\frac{\epsilon^2}{2}\hat\pi^\alpha_s)dw_s-\frac{1}{2}\int_0^t| \epsilon\pi^{\alpha}_s+\frac{\epsilon}{2}\hat \pi^\alpha_s|^2ds\bigg).
$$
with $\alpha\in\mathfrak A$, $r_s^{\alpha}, \pi_s^{\alpha}, P_s^{\alpha}, \hat r_s^{\alpha}, \hat\pi_s^{\alpha}, \hat P_s^{\alpha}$ defined in Lemma \ref{3l1}.

By intruducing
\begin{align*}
&\hat q_t^\alpha(\epsilon)=\int_0^{t}(1+2\epsilon r^\alpha_s+\epsilon^2\hat r^\alpha_s)f^{\alpha_s}(z_s( \epsilon))\hat p_s( \epsilon)e^{-\hat\phi_s( \epsilon)}ds,\\
&\bar z_t^{\alpha,z}(\epsilon)=(z_t^{\alpha,z}(\epsilon),-\hat\phi_t^\alpha(\epsilon),\hat p_t^\alpha(\epsilon),\hat q_t^\alpha(\epsilon)),\\
&\bar x_t^{\alpha,x}=(x_t^{\alpha,x},-\hat\phi_t^\alpha(0),\hat p_t^\alpha(0),\hat q_t^\alpha(0)),
\end{align*}
we get
\begin{align*}
&-\frac{v(x+ \epsilon\xi)-2v(x)+v(x- \epsilon\xi)}{ \epsilon^2}\\
=&\frac{1}{\epsilon^2}\Big(-\sup_{\alpha\in\mathfrak A}E^\alpha_{x+\epsilon\xi}V(\bar z_\gamma(\epsilon)+2\sup_{\alpha\in\mathfrak A}E^\alpha_xV(\bar x_\gamma)-\sup_{\alpha\in\mathfrak A}E^\alpha_{x-\epsilon\xi}V(\bar z_\gamma(-\epsilon)\Big)\\
\le&\sup_{\alpha\in\mathfrak A}\frac{-E^\alpha_{x+\epsilon\xi}V(\bar z_\gamma(\epsilon)+2E^\alpha_x V(\bar x_\gamma)- E^\alpha_{x-\epsilon\xi}V(\bar z_\gamma(-\epsilon)}{\epsilon^2}\\
=&\sup_{\alpha\in\mathfrak A}E\frac{-V(\bar z^{\alpha,x+\epsilon\xi}_{\gamma^\alpha}(\epsilon)+2 V(\bar x^{\alpha,x}_{\gamma^\alpha})- V(\bar z^{\alpha,x-\epsilon\xi}_{\gamma^\alpha}(-\epsilon))}{\epsilon^2}\\
\le&\sup_{\alpha\in\mathfrak A}E\bigg[\frac{-V(\bar z^{\alpha,x+\epsilon\xi}_{\gamma^\alpha}(\epsilon)+2 V(\bar x^{\alpha,x}_{\gamma^\alpha})- V(\bar z^{\alpha,x-\epsilon\xi}_{\gamma^\alpha}(-\epsilon))}{\epsilon^2}+V_{(\bar \eta^{\alpha,0}_{\gamma^\alpha})}(\bar x^{\alpha,x}_{\gamma^\alpha})+V_{(\bar \xi^{\alpha,\xi}_{\gamma^\alpha})(\bar \xi^{\alpha,\xi}_{\gamma^\alpha})}(\bar x^{\alpha,x}_{\gamma^\alpha})\bigg]\\
&+\sup_{\alpha\in\mathfrak A}E\Big[-V_{(\bar \eta^{\alpha,0}_{\gamma^\alpha})}(\bar x^{\alpha,x}_{\gamma^\alpha})-V_{(\bar \xi^{\alpha,\xi}_{\gamma^\alpha})(\bar \xi^{\alpha,\xi}_{\gamma^\alpha})}(\bar x^{\alpha,x}_{\gamma^\alpha})\Big]\\
:=&G_1(\epsilon, n,T)+G_2(\epsilon,n,T),
\end{align*}
where $V$ and $\bar\xi^{\alpha,\xi}_t$ are defined by (\ref{V}) and (\ref{xibar}), respectively, and
$$\bar\eta^{\alpha,\eta}_t:=(\eta_t^{\alpha,\eta}, \eta^{d+1,\alpha}_t, \eta^{d+2,\alpha}_t, \eta^{d+3,\alpha}_t),$$
with
\begin{align*}
\eta^{d+1,\alpha}_t:=&-\int_0^tc_{(\xi_s^{\alpha,\xi})(\xi_s^{\alpha,\xi})}(x_s^{\alpha,x})+c_{(\eta_s^{\alpha,\eta})}(x_s^{\alpha,s})+4r^\alpha_sc_{(\xi_s^{\alpha,\xi})}(x_s^{\alpha,x})+2\hat r^\alpha_sc(x_s^{\alpha,x})ds,\\
\eta^{d+2,\alpha}_t:=&\eta^{0,\alpha}_t=\Big(\int_0^t\pi_s^\alpha dw_s\Big)^2-\int_0^t|\pi^\alpha_s|^2ds+\int_0^t\hat\pi_sdw_s,\\
\eta^{d+3,\alpha}_t:=&\int_0^te^{-\phi_s^{\alpha,x}}\Big[f^{\alpha_s}_{(\xi_s^{\alpha,\xi})(\xi_s^{\alpha,\xi})}(x_s^{\alpha,x})+f_{(\eta_s^{\alpha,\eta})}(x_s^{\alpha,x})+(2\xi_s^{d+1,\alpha}+4r^\alpha_s)f^{\alpha_s}_{(\xi_s^{\alpha,\xi})}(x_s^{\alpha,x})\\
&+\big((\xi_s^{d+1,\alpha})^2+\eta_s^{d+1,\alpha}+4r^\alpha_s\xi_s^{d+1,\alpha}+2\hat r^\alpha_s\big) f^{\alpha_s}(x_s^{\alpha,x})\Big]ds.
\end{align*}

We first claim that
$$\lim_{\epsilon\downarrow0}G_1(\epsilon, n,T)=0.$$
The proof is similar as that of (\ref{I1}) with the help of the following two second-order counterparts.

First, if $h^\alpha(x)\in C^2(\bar D_\delta)$, and the derivatives of $h^\alpha(x)$ are uniformly continuous in $\alpha$, then for any $x,z,z'\in D_\delta$, $\xi,\eta\in\Rd$, $r, \hat r\in \mathbb R$ and $n\in\mathbb N$, we have
\begin{align*}
&\frac{h^\alpha(z)-2h^\alpha(x)+h^\alpha(z')}{\epsilon^2}\\
=&\frac{1}{\epsilon^2}\Big[h^\alpha_{(z-x)}(x)+\frac{1}{2}h^\alpha_{(z-x)(z-x)}(z^*)+h^\alpha_{(z'-x)}(x)+\frac{1}{2}h^\alpha_{(z'-x)(z'-x)}({z^*}')\Big]\\
=&h^\alpha_{(\frac{z-2x+z'}{\epsilon^2})}(x)+\frac{1}{2}\Big[h^\alpha_{(\frac{z-x}{\epsilon})(\frac{z-x}{\epsilon})}(z^*)+h^\alpha_{(\frac{z'-x}{\epsilon})(\frac{z'-x}{\epsilon})}({z^*}')\Big],
\end{align*}
where $z^*$ and ${z^*}'$ are on the line segments $\overline{xz}$ and $\overline{xz'}$, respectively. Hence,
\begin{align*}
&|\frac{(1+2\epsilon r+\epsilon^2\hat r)h^\alpha(z)-2h^\alpha(x)+(1-2\epsilon r+\epsilon^2\hat r)h^\alpha(z')}{\epsilon^2}\\
&-(h^\alpha_{(\xi)(\xi)}(x)+h^\alpha_{(\eta)}(x)+4rh^\alpha_{(\xi)}(x)+2\hat rh^\alpha(x))|\\
\le&|\frac{h^\alpha(z)-2h^\alpha(x)+h^\alpha(z')}{\epsilon^2}-(h^\alpha_{(\xi)(\xi)}(x)+h^\alpha_{(\eta)}(x))|\\
&+2|r||\frac{h^\alpha(z)-h^\alpha(z')}{\epsilon}-2h^\alpha_{(\xi)}(x)|+|\hat r||h^\alpha(z)+h^\alpha(z')-2h^\alpha(x)|\\
\le&|h^\alpha_{(\frac{z-2x+z'}{\epsilon^2}-\eta)}(x)|+\frac{1}{2}\Big[|h^\alpha_{(\frac{z-x}{\epsilon})(\frac{z-x}{\epsilon})}(z^*)-h^\alpha_{(\xi)(\xi)}(x)|+|h^\alpha_{(\frac{z'-x}{\epsilon})(\frac{z'-x}{\epsilon})}({z^*}')-h^\alpha_{(\xi)(\xi)}(x)|\Big]\\
&+2|r|\Big[|\frac{h^\alpha(z)-h^\alpha(x)}{\epsilon}-h^\alpha_{(\xi)}(x)|+|\frac{h^\alpha(z')-h^\alpha(x)}{-\epsilon}-h^\alpha_{(\xi)}(x)|\Big]\\
&+|\hat r|\Big[|h^\alpha(z)-h^\alpha(x)|+|h^\alpha(z')-h^\alpha(x)|\Big].
\end{align*}

Second, by noticing that
$$\frac{\hat p_t(\epsilon)-2\hat p_t(0)+\hat p_t(-\epsilon)}{\epsilon^2}=\int_0^t\Big(\frac{\hat p_s(\epsilon)-\hat p_s(-\epsilon)}{\epsilon}\pi_s+\frac{\hat p_s(\epsilon)+\hat p_s(-\epsilon)}{2}\hat \pi_s\Big)dw_s,$$
$$\eta_t^{d+2}=\eta_t^0=\int_0^t(2\xi_s^0\pi_s+\hat\pi_x)dw_s,$$
we have
\begin{align*}
&E^\alpha\sup_{t\le\gamma}|\frac{\hat p_t(\epsilon)-2\hat p_t(0)+\hat p_t(-\epsilon)}{\epsilon^2}-\eta_t^{d+2}|^p\\
\le&N(p)E^\alpha\bigg(\int_0^\gamma \Big(\frac{\hat p_t(\epsilon)-\hat p_t(-\epsilon)}{\epsilon}-2\xi_t^0\Big)^2|\pi_t|^2+\Big(\frac{\hat p_t(\epsilon)+\hat p_t(-\epsilon)}{2}-1\Big)^2|\hat\pi_t|^2dt\bigg)^{p/2}\\
\le&N(p)E^\alpha\bigg(\epsilon^{-p}\sup_{t\le\gamma}|\frac{\hat p_t(\epsilon)-\hat p_t(-\epsilon)}{\epsilon}-2\xi_t^0|^{2p}+\epsilon^{-p}\sup_{t\le\gamma}|\frac{\hat p_t(\epsilon)+\hat p_t(-\epsilon)}{2}-1|^{2p}\\
&+\epsilon^p\int_0^\gamma|\pi_t|^{2p}dt+\epsilon^p\int_0^\gamma|\hat\pi_t|^{2p}dt\bigg)\\
\le&\epsilon^pN(p)E^\alpha\bigg(\int_0^\gamma \hat p^{2p}_t(\epsilon)|\pi_t+\frac{\epsilon}{2}\hat\pi_t|^{2p}dt+\int_0^\gamma\hat p^{2p}_t(-\epsilon)|\pi_t+\frac{\epsilon}{2}\hat\pi_t|^{2p}dt\\
&+\int_0^\gamma|\pi_t+\frac{\epsilon}{2}\hat\pi_t|^{2p}dt+\int_0^\gamma|\pi_t|^{2p}dt+\int_0^\gamma|\hat\pi_t|^{2p}dt\bigg).
\end{align*}
Therefore, 
$$\lim_{\epsilon\downarrow0}\sup_{\alpha\in\mathfrak A}E^\alpha\sup_{t\le\gamma}|\frac{\hat p_t(\epsilon)-2\hat p_t(0)+\hat p_t(-\epsilon)}{\epsilon^2}-\eta_t^{d+2}|^p=0.$$

In order to estimate $G_2(\epsilon,n, T)$, we notice that $V_{(\bar\eta^{\alpha,0}_t)}(\bar x^{\alpha,x}_t)+V_{(\bar\xi^{\alpha,\xi}_t)(\bar\xi^{\alpha,\xi}_t)}(\bar x^{\alpha,x}_t)$ is exactly $Y_t^\alpha$ defined by (2.10) in \cite{quasi-linear}, in which $u$ is replaced by $v$, that is
\begin{align*}
&V_{(\bar\eta^{\alpha,0}_t)(\bar x^{\alpha,x}_t)}+V_{(\bar\xi^{\alpha,\xi}_t)(\bar\xi^{\alpha,\xi}_t)}(\bar x^{\alpha,x}_t)=Y_t^\alpha\\
:=&e^{-\phi_t^{\alpha,x}}\Big[v_{(\xi_t^{\alpha,\xi})(\xi_t^{\alpha,\xi})}(x_t^{\alpha,x})+v_{(\eta_t^{\alpha,0})}(x_t^{\alpha,x})+2\tilde{\xi}_t^0v_{(\xi_t^{\alpha,\xi})}(x_t^{\alpha,x})+\tilde{\eta}_t^0v(x_t^{\alpha,x})\Big]\\
&+\int_0^te^{-\phi_s^{\alpha,x}}\Big[f^{\alpha_s}_{(\xi_s^{\alpha,x})(\xi_s^{\alpha,\xi})}(x_s^{\alpha,x})+f^{\alpha_s}_{(\eta_s^{\alpha,0})}(x_s^{\alpha,x})+\big(4r^\alpha_s+2\tilde{\xi}_s^0\big)f^{\alpha_s}_{(\xi_s^{\alpha,\xi})}(x_s^{\alpha,x})\\
&\qquad\qquad\qquad+\big(2\hat{r}^\alpha_s+4\tilde{\xi}_s^0r^\alpha_s+\tilde{\eta}_s^0\big)f^{\alpha_s}(x_s)\Big]ds,
\end{align*}
where
\begin{align*}
\tilde{\eta}_t^0=\eta_t^{d+2}+2\xi_t^{d+2}\xi_t^{d+1}+(\xi_t^{d+1})^2+\eta_t^{d+1}.
\end{align*}

As in the proof of (3.5) in \cite{quasi-linear}, for each $\alpha$,
\begin{align*}
&E\sup_{t\le\tau^{\alpha,x}_{D_\delta^\lambda}}\big(Y^\alpha_t-e^{-\phi_t^{\alpha,x}}v_{(\xi_t^{\alpha,\xi})(\xi_t^{\alpha,\xi})}(x_t^{\alpha,x})\big)\\
=&e^{-\phi_t^{\alpha,x}}\Big[v_{(\eta_t^{\alpha,0})}(x_t^{\alpha,x})+2\tilde{\xi}_t^0v_{(\xi_t^{\alpha,\xi})}(x_t^{\alpha,x})+\tilde{\eta}_t^0v(x_t^{\alpha,x})\Big]\\
&+\int_0^te^{-\phi_s^{\alpha,x}}\Big[f^{\alpha_s}_{(\xi_s^{\alpha,x})(\xi_s^{\alpha,\xi})}(x_s^{\alpha,x})+f^{\alpha_s}_{(\eta_s^{\alpha,0})}(x_s^{\alpha,x})+\big(4r^\alpha_s+2\tilde{\xi}_s^0\big)f^{\alpha_s}_{(\xi_s^{\alpha,\xi})}(x_s^{\alpha,x})\\
&\qquad\qquad\qquad+\big(2\hat{r}^\alpha_s+4\tilde{\xi}_s^0r^\alpha_s+\tilde{\eta}_s^0\big)f^{\alpha_s}(x_s)\Big]ds\\
\le&N\Big(|g|_{0,D}+|f^\alpha|_{2,D}+\sup_{x\in\partial D_\delta^\lambda,|\zeta|=1}|v_{(\zeta)}(x)|\Big)\\
&\cdot\Big(E\sup_{t\le\tau^{\alpha,x}_{D_\delta^\lambda}}|\eta_t|+E\sup_{t\le\tau^{\alpha,x}_{D_\delta^\lambda}}|\xi_t|^2+E\sup_{t\le\tau^{\alpha,x}_{D_\delta^\lambda}}|\xi^0_t|^2+ E\sup_{t\le\tau^{\alpha,x}_{D_\delta^\lambda}}e^{-\frac{1}{2}t}|\xi_t^{d+1}|^2\\
&\ \ \ \ +E\sup_{t\le\tau^{\alpha,x}_{D_\delta^\lambda}}e^{-\frac{1}{2}t}|\eta_t^{d+1}|+E\int_0^{\tau^{\alpha,x}_{D_\delta^\lambda}}r_s^2+\hat{r}_sds\Big)
\end{align*}
where $N$ is independent of $\alpha$.
Repeat the estimates (3.30)-(3.35) in \cite{quasi-linear}, we have
$$E\sup_{t\le\tau^{\alpha,x}_{D_\delta^\lambda}}\big(Y^\alpha_t-e^{-\phi_t^{\alpha,x}}v_{(\xi_t^{\alpha,\xi})(\xi_t^{\alpha,\xi})}(x_t^{\alpha,x})\big)\le N_1\mathrm B_1(x,\xi),$$
with
$$N_1=N\Big(|g|_{2,D}+\sup_{\alpha}|f^\alpha|_{2,D}+\sup_{x\in\partial D_\delta^\lambda,|\zeta|=1}|v_{(\zeta)}(x)|\Big),$$
where $N$ is independent of $\alpha$. Hence
$$G_2(\epsilon,T)\le\sup_{\alpha\in\mathfrak A}E\Big(-e^{-\phi_t^{\alpha,x}}v_{(\xi_t^{\alpha,\xi})(\xi_t^{\alpha,\xi})}(x_t^{\alpha,x})\Big)+N_1\mathrm B_1(x,\xi).$$

By mimicking the argument in the proof of (\ref{4d}), we have
$$\varlimsup_{T\uparrow\infty}\varlimsup_{\epsilon\downarrow0}\sup_{\alpha\in\mathfrak A}E\Big(-e^{-\phi_t^{\alpha,x}}v_{(\xi_t^{\alpha,\xi})(\xi_t^{\alpha,\xi})}(x_t^{\alpha,x})\Big)\le\bigg(\sup_{y\in\partial D_\delta^\lambda, \zeta\in\Rd\setminus\{0\}}\frac{(-v)_{(\zeta)(\zeta)}(y)_+}{\mathrm B_1(x,\zeta)}+3\bigg)\mathrm B_1(x,\xi),$$
where
$$(-v)_{(\zeta)(\zeta)}(y)_+=(-v)_{(\zeta)(\zeta)}(y)\vee0.$$

So we conclude that
$$\varlimsup_{T\uparrow\infty}\varlimsup_{\epsilon\downarrow0}G_2(\epsilon, T)\le\sup_{y\in\partial D_\delta^\lambda,\zeta\in\Rd\setminus\{0\}}\frac{(-v)_{(\zeta)(\zeta)}(y)_+}{\mathrm B_1(y,\zeta)}\cdot\mathrm B_1(x,\xi)+N_1\mathrm B_1(x,\xi),$$
which implies that
\begin{equation}\label{ind1'}
\frac{(-v)_{(\xi)(\xi)}(x)_+}{\mathrm B_1(x,\xi)}\le\sup_{y\in\partial D_\delta^\lambda,\zeta\in\Rd\setminus\{0\}}\frac{(-v)_{(\zeta)(\zeta)}(y)_+}{\mathrm B_1(y,\zeta)}+N_1,\ \forall x\in D_{\delta}^\lambda,\xi\in\Rd\setminus\{0\}.
\end{equation}
Repeating the argument above for $D_{\lambda^2}$, we have
\begin{equation}\label{ind2'}
\frac{(-v)_{(\xi)(\xi)}(x)_+}{\mathrm B_1(x,\xi)}\le\sup_{y\in\partial D_{\lambda^2},\zeta\in\Rd\setminus\{0\}}\frac{(-v)_{(\zeta)(\zeta)}(y)_+}{\mathrm B_1(y,\zeta)}+N_1,\ \forall x\in D_{\lambda^2},\xi\in\Rd\setminus\{0\}.
\end{equation}
Since (\ref{ind1'}) and (\ref{ind2'}) are similar as (3.36) and (3.38) in \cite{quasi-linear}, by repeating the argument after (3.38) in \cite{quasi-linear}, we get
$$
(-v)_{(\xi)(\xi)}(x)_+\le N\bigg(|\xi|^2+\frac{\psi_{(\xi)}^2(x)}{\psi(x)}\bigg), \ \ \mbox{ a.e. in }D.
$$
The inequality (\ref{4dd}) is proved.
\end{proof}

\begin{proof}[Proof of (\ref{4ddd})]

Fix an $x\in D$. For simplicity of notation we will drop the argument $x$ through the proof below.

From (\ref{4dd}) we have
$$v_{(\xi)(\xi)}+N\bigg(|\xi|^2+\frac{\psi_{(\xi)}^2}{\psi}\bigg)\ge 0, \forall\xi\in\Rd.$$
It follows that
$$v_{(\xi)(\xi)}+\frac{N}{\psi}|\xi|^2\ge0, \forall\xi\in\Rd.$$
Let
$$
V=v_{xx}+(\frac{N}{\psi}+1)I,
$$
where $I$ is the identity matrix of size $d\times d$.

Then we have
$$
\big(V\xi,\xi\big)\ge|\xi|^2>0, \forall\xi\in\Rd\setminus\{0\}.
$$
Fix a $\xi\in\Rd$ such that $\mu(\xi)>0$. Introduce
$$
\kappa=\sqrt{V}\xi, \qquad \theta=|\kappa|^{-2}\kappa,\qquad \zeta=\sqrt{V}\theta.
$$
Then
\begin{align*}
\mathrm{tr}(a^\alpha V)&=\mathrm{tr }(\sqrt{V}a^\alpha \sqrt{V})\\
&\ge|\theta|^{-2}(\sqrt{V}a^\alpha \sqrt{V}\theta,\theta)=|\kappa|^2(a^\alpha\zeta,\zeta)=(V\xi,\xi)(a^\alpha\zeta,\zeta).
\end{align*}
Taking the supremum and noticing that $(\xi,\zeta)=(\kappa,\theta)=1$, we get
$$\sup_{\alpha\in A}\mathrm{tr}(a^\alpha V)\ge (V\xi,\xi)\sup_{\alpha\in A}(a^\alpha\zeta,\zeta)\ge(V\xi,\xi)\mu(\xi).$$
It follows that
\begin{align*}
v_{(\xi)(\xi)}\le (V\xi,\xi)&\le \mu^{-1}(\xi)\sup_{\alpha\in A}\mathrm{tr}(a^\alpha V)\\
&\le\mu^{-1}(\xi)\Big[\sup_{\alpha\in A}\mathrm{tr}(a^\alpha v_{xx})+\frac{N}{\psi}\sup_{\alpha\in A}\mathrm{tr}(a^\alpha)\Big].
\end{align*}
Notice that
$$\mu(\xi)=|\xi|^{-2}\mu(\xi/|\xi|),$$
so it remains to estimate $\sup_{\alpha\in A}\mathrm{tr}(a^\alpha v_{xx})$ from above.
The equation
$$\sup_{\alpha\in A}\big[L^\alpha v-c^\alpha v+f^\alpha\big]=0$$
implies that
$$L^\alpha v-c^\alpha v+f^\alpha\le0, \forall \alpha\in A.$$
Thus
$$\mathrm{tr}(a^\alpha v_{xx})=(a^\alpha)^{ij}v_{x^ix^j}\le|(b^\alpha)^i|_{0,D}|v_{x^i}|_{0,D}+|c^\alpha|_{0,D}|v|_{0,D}+|f^\alpha|_{0,D}\le K.$$

\end{proof}

\begin{proof}[Proof of the existence and uniqueness of (\ref{bellmanae})]
The fact that $v$ given by (\ref{v}) and (\ref{vax}) satisfies (\ref{bellmanae}) follows from Theorem 1.3 in \cite{MR617995}.

To proof the uniqueness, assume that $v_1, v_2\in C_{loc}^{1,1}(D)\cap C^{0,1}(\bar D)$ are solutions of (\ref{bellmanae}).  Let $\Lambda=|v_1|_{0,D}\vee|v_2|_{0,D}$. For constants $\delta$ and $\ve$ satisfying $0<\delta<\ve<1$, define
$$\Psi(x,t)=\ve(1+\psi(x))\Lambda e^{-\delta t},\ V(x,t)=v(x) e^{-\ve t}\mbox{ in }\bar D\times(0,\infty),$$
$$F[V]=\sup_{\alpha\in A}(V_t+L^\alpha V-c^\alpha V+f^\alpha) \mbox{ in } D\times (0,\infty).$$

Notice that a.e. in $D$, we have
$$F[V_1-\Psi]\ge-\ve e^{-\ve t}v_1+\delta\Psi-\ve\Lambda e^{-\delta t}\sup_{\alpha}L^\alpha\psi+\inf_{\alpha}c^\alpha\Psi\ge\ve\Lambda(e^{-\delta t}-e^{-\ve t})\ge0, $$
$$F[V_2+\Psi]\le\ve e^{-\ve t}v_2-\delta\Psi+\ve\Lambda e^{-\delta t}\sup_{\alpha}L^\alpha\psi-\inf_\alpha c^\alpha\Psi\le\ve\Lambda(e^{-\ve t}-e^{-\delta t})\le0.$$
On $\partial D\times (0,\infty)$, we have
$$V_1-V_2-2\Psi=-2\Psi\le0.$$
On $\bar D\times T$, where $T=T(\ve,\delta)$ is a sufficiently large constant, we have
$$V_1-V_2-2\Psi=(v_1-v_2)e^{-\ve T}-2\ve(1+\psi)\Lambda e^{-\delta T}\le2\Lambda(e^{-\ve T}-\ve e^{-\delta T})\le 0.$$

Applying Theorem 1.1 in \cite{MR538554}, we get
$$V_1-V_2-2\Psi\le 0 \mbox { a.e. in } \bar D\times(0,T).$$
It follows that
$$v_1-v_2\le 2\ve(1+\psi)\Lambda e\rightarrow0, \mbox{ as } \ve\rightarrow0, \mbox{ a.e. in }D.$$
Similarly, $v_2-v_1\le0$ a.e. in $D$. The uniqueness is proved.

\end{proof}

\section*{Acknowledgement} The author is sincerely grateful to his advisor, N. V. Krylov, for giving many useful suggestions on the improvements. The author also would like to thank the referee for pointing out several misprints and mistakes and giving comments on the manuscript of this article.


\end{document}